\documentclass[12pt]{amsart}

\usepackage{amsmath}
\usepackage{amssymb}
\usepackage{amsfonts}
\usepackage{amsthm}
\usepackage{enumerate}
\usepackage{color}
\usepackage{psfrag}
\usepackage{mathrsfs}
\usepackage{amscd}
\usepackage[hyperindex,backref=page]{hyperref}
\usepackage{graphicx,color}
\usepackage{tikz} 
\usepackage{array}
\usepackage{float}
\usepackage{longtable}
\usepackage{caption}
\usepackage{enumitem}

\DeclareUnicodeCharacter{200B}{}

\usetikzlibrary{positioning}
\usetikzlibrary{calc}

\newtheorem{thm}{Theorem}[section]
\newtheorem{cor}[thm]{Corollary}
\newtheorem{lemma}[thm]{Lemma}
\newtheorem{prop}[thm]{Proposition}

\theoremstyle{definition}
\newtheorem{defi}[thm]{Definition}
\newtheorem{ex}[thm]{Example}
\newtheorem{nota}[thm]{Notation}

\theoremstyle{remark}
\newtheorem{rem}[thm]{Remark}

\dedicatory{To the memory of Professor J\"urgen Herzog}

\begin{document}


\title[rooted product of graphs and multi-clique corona graphs]{Algebraic study on rooted products of graphs and multi-clique corona graphs}

\author{Yuji Muta and Naoki Terai}

\address[Y. Muta]{Department of Mathematics, Okayama University, 3-1-1 Tsushima-naka, Kita-ku, Okayama 700-8530 Japan.}

\email{p8w80ole@s.okayama-u.ac.jp}

\address[N. Terai]{Department of Mathematics, Okayama University, 3-1-1 Tsushima-naka, Kita-ku, Okayama 700-8530, Japan.}
	
\email{terai@okayama-u.ac.jp}


\begin{abstract}
In this paper, we study rooted products of graphs from the perspective of combinatorial commutative algebra. For edge ideals, we introduce the 2-Cohen--Macaulayness with respect to a vertex and use it to investigate when edge ideals of rooted products of graphs are Cohen--Macaulay. Moreover, we completely determine when attaching a graph on at most six vertices to a given graph as rooted products, yields a Cohen--Macaulay edge ideal. Also, we define mulit-clique corona graphs as a generalization of clique-corona graphs and multi-whisker graphs. We prove that multi-clique corona graphs are vertex decomposable and hence sequentially Cohen--Macaulay. Also, we give formulas for the projective dimension and the Castelnuovo--Mumford regularity. 
\end{abstract}






\thanks{}
	


\maketitle

\section{Introduction}

Studying the edge ideals of finite simple graphs is one of the central topics in combinatorial commutative algebra. In \cite{v1}, Villarreal introduced edge ideals of finite simple graphs and characterized precisely when the edge ideal of a tree is Cohen--Macaulay. Simis, Vasconcelos, and Villarreal then set up a ring-theoretic framework for edge ideals \cite{svv}. Building on this, Herzog and Hibi gave a complete characterization of Cohen--Macaulay bipartite graphs \cite{hh}, and Herzog, Hibi, and Zheng subsequently characterized the chordal case \cite{hhz}. Taken together, these contributions have provided the foundation for ongoing research on the ring-theoretic properties of edge ideals (\cite{cdkkv, cn, crt, fh, ft, hhm, hhko, hmt, hp, ht, mpt, n, t2, t3, vv}).

Throughout, a graph means a finite undirected graph without loops and multiple edges. Let $G=(V(G), E(G))$ be a graph, where $V(G)=\{x_{1},\ldots, x_{n}\}$ is the vertex set of $G$ and $E(G)$ is the edge set of $G$ and let $\Bbbk$ be an arbitrary field and $S=\Bbbk[x_{1},\ldots, x_{n}]$ be a polynomial ring over the field $\Bbbk$. Then the ideal of $S$ $$I(G)=(x_{i}x_{j}\,\,: \{x_{i},x_{j}\}\in E(G))$$ is called the {\it edge ideal} of $G$. Villarreal showed that attaching a whisker to every vertex of a graph makes its edge ideal Cohen--Macaulay \cite[Proposition 2.1]{v1}. This has been generalized in several directions, including clique-whiskered graphs \cite{cn}, clique corona graphs \cite{hp}, multi-whisker graphs \cite{mpt}, and multi-clique-whiskered graphs \cite{hhm, mt2}. 

In the first part of this paper, we investigate the edge ideals of rooted products of graphs. The edge ideals of rooted products of graphs have been studied in \cite{mfy} . Let us recall the definition of rooted products of graphs, which was introduced by Godsil and McKay in \cite[Definition 1.1]{gm}. Let $H_{1}, \ldots, H_{m}$ be graphs with $V(H_{i})\cap V(H_{j})=\emptyset$ for all $i\neq j$, and fix a vertex $x_{i} \in V(H_{i})$ for each $i = 1, \ldots, m$. Let $G_0$ be a graph on the vertex set $X_{[m]}=\{x_{1}, \ldots, x_{m}\}$, and set $\mathcal{H} = \{H_{1}, \ldots, H_{m}\}$. Then, the graph $G_{0}(\mathcal{H})$, whose vertex set is $V(G_{0}(\mathcal{H}))=\bigcup_{1\leq i\leq m}V(H_{i})$ and whose edge set is $E(G_{0}(\mathcal{H}))=E(G_{0})\cup\left(\bigcup_{1\leq i\leq m}E(H_{i})\right)$ is called the rooted product of $G_{0}$ by $\mathcal{H}$. We always assume that the following conditions hold for $G_{0}(\mathcal{H})$:
\begin{itemize}
\item[$(\ast)$] $m\geq 2$ and $H_{i}$ is connected for all $i$ and $G_{0}(\mathcal{H})$ is not the empty graph. 
\end{itemize} 
We now turn to the corona product, a construction closely related to the rooted product. For a graph $G$ on the vertex set $X_{[h]}=\{x_{1},\ldots, x_{h}\}$, we set $\mathcal{H}=\{H_{i}\,\,: x_{i}\in V(G)\},$ where $H_{i}$ is a non-empty graph indexed by the vertex $x_{i}$. The {\it corona graph} $G\circ\mathcal{H}$ of $G$ and $\mathcal{H}$ is the disjoint union of $G$ and $H_{i}$, with additional edges joining each vertex $x_{i}$ to all vertices $H_{i}$. In particular, $G\circ\mathcal{H}$ is called the {\it clique corona graph}, if $H_{i}$ is the complete graph for all $i$. 

In the corona construction, at each vertex of the base graph one attaches a graph and connects that vertex by edges to every vertex of the attached graph. By contrast, in the rooted product, the attachment is by identification only, and no additional edges are added between the base graph and the attached graphs. Notice that corona graphs can be expressed as rooted products of graphs. 

Hoang and Pham proved that a corona graph is Cohen--Macaulay if and only if it is a clique corona graph \cite[Theorem 2.6]{hp}. Accordingly, it is natural to ask which graphs can be attached to a given graph as rooted products, yield a graph whose the edge ideal is Cohen--Macaulay. In view of this, we introduce the following notion: 

\begin{defi}
Let $G$ be a graph, $x$ be a vertex of $G$ and let $r\geq2$. Then, we call $G$ satisfies {\it 2-Serre's condition $(S_{r})$ with respect to $x$}, if the following conditions are satisfied: 
\begin{enumerate}
\item $G$ satisfies Serre's condition $(S_{r})$. 
\item $G-x$ satisfies Serre's condition $(S_{r})$ with $\dim S/I(G)=\dim S/(I(G)+x)$. 
\end{enumerate}
If $G$ satisfies 2-Serre's condition $(S_{r})$ for all vertices, then we call $G$ satisfies {\it 2-Serre's condition $(S_{r})$}.  

In particular, we call $G$ is {\it 2-Cohen--Macaulay with respect to $x$}, if the following conditions are satisfied: 
\begin{enumerate}
\item $G$ is Cohen--Macaulay. 
\item $G-x$ is Cohen--Macaulay with $\dim S/I(G)=\dim S/(I(G)+x)$. 
\end{enumerate}
If $G$ is 2-Cohen--Macaulay for all vertices, then $G$ is called  {\it 2-Cohen--Macaulay.} 

\end{defi}

Using this notion, we prove the following theorem, which characterizes when the edge ideals of rooted products of graphs satisfy Serre's condition $(S_{r})$ and, in particular, when they are Cohen--Macaulay:

\begin{thm}[\mbox{\rm see, Theorem \ref{Serre}}]
In the setting above and assume that the condition $(\ast)$ holds and let $r\geq2$. Then, the following conditions are equivalent: 
\begin{enumerate}
\item $H_{i}$ satisfies 2-Serre's condition $(S_{r})$ with respect to $x_{i}$ for all $i$. 
\item $G_{0}(\mathcal{H})$ satisfies Serre's condition $(S_{r})$ for every graph $G_{0}$ on the vertex set $X_{[m]}$. 
\end{enumerate}
Moreover, under the assumption that $|V(H_{i})| \geq 2$ for all $i$, the following condition is equivalent to the conditions above:
\begin{enumerate}[start=3]
\item $G(\mathcal{H})$ satisfies Serre's condition $(S_{r})$ for every graph $G$ in $\mathcal{G}$, 
\end{enumerate}
where $\mathcal{G}$ is a family of graphs on the vertex set $X_{[m]}$ which satisfies the following condition: 

\begin{center}
For any $i=1,\ldots, m$, there exists $G\in\mathcal{G}$ such that $x_{i}$ is not an isolated vertex in $G$. 
\end{center}
\end{thm}

As a corollary, we obtain that the following characterization of $G_{0}(\mathcal{H})$:
\begin{cor}[\mbox{\rm see, Corollaries \ref{cor about Serre}, \ref{cor about CM}}]\label{Intro of CM}
In the setting above and assume that the condition $(\ast)$ holds and let $r\geq2$.
Suppose that $|V(H_{i})|\geq2$ for all $i$. Fix a graph $G_{0}$ on the vertex set $X_{[m]}$ without isolated vertices. Then, the following conditions are equivalent: 
\begin{enumerate}
\item $H_{i}$ satisfies 2-Serre's condition $(S_{r})$ with respect to $x_{i}$ for all $i$. 
\item $G_{0}(\mathcal{H})$ satisfies Serre's condition $(S_{r})$. 
\end{enumerate}
In particular, the following conditions are equivalent: 
\begin{enumerate}
\item $H_{i}$ is 2-Cohen--Macaulay with respect to $x_{i}$ for all $i$. 
\item$G_{0}(\mathcal{H})$ is Cohen--Macaulay. 
\end{enumerate}
\end{cor}

By Corollary \ref{Intro of CM}, the Cohen--Macaulayness of the edge ideal may be determined by decomposing a graph. Moreover, we completely determine when attaching a graph on at most six vertices to a given graph as rooted products, yields a Cohen--Macaulay edge ideal by using the lists given by Villarreal \cite[pp. 635--637]{v2}. Also, the equivalence between being pure vertex decomposable and being Cohen--Macaulay is known to hold in several classes of graphs; for example, well-covered chordal graphs \cite{hhz}, very well-covered graphs \cite[Theorem 1.1]{mmcrty}, Cameron--Walker graphs \cite[Theorem 1.3]{hhko}, permutation graphs \cite[Theorem 2.3]{fm}. We also prove that this equivalence holds for rooted products of graphs under certain conditions on $H_{1},\ldots,H_{m}$. (see, Theorem \ref{eqiv between CM and vd})

In the second part of this paper, we study a generalization of clique corona graphs and multi-whisker graphs. Pournaki and the authors of this paper introduced multi-whisker graphs as a generalization of whisker graphs by attaching some edges to each vertex \cite{mpt}. 

As a further generalization of these graphs, we define a multi-clique corona graphs by attaching some complete graphs to each vertex, namely, for a graph $G$ on the vertex set $X_{[h]}=\{x_{1},\ldots, x_{h}\}$ and positive integers $n_{1},\ldots, n_{h}$, also, we set $\mathcal{H}=\cup_{1\leq i\leq m}\{H_{i}^{1},\ldots, H_{i}^{n_{i}}\},$ where $H_{i}^{j}$ is a nonempty graph indexed by the vertex $x_{i}$. 
we define a {\it multi-corona graph} $G\circ\mathcal{H}$ of $G$ and $\mathcal{H}$ is the disjoint union of $G$ and $H_{i}^{1},\ldots, H_{i}^{n_{i}}$, with additional edges joining each vertex $x_{i}$ to all vertices $H_{i}^{j}$. We call $G\circ\mathcal{H}$ is a {\it multi-clique corona graph}, if $H_{i}^{j}$ is the complete graph for all $i$ and $j$. 

We prove that, after an appropriate transformation of a multi-clique corona graph into a multi-whisker graph, the Betti numbers of the edge ideal coincide with those of the edge ideal of a multi-whisker graph (see, Theorem \ref{betti}). Consequently, we obtain the following theorem as generalizations of results of multi-whisker graphs: 

\begin{thm}[\mbox{\rm see, Corollaries \ref{reg}, \ref{pd} and Theorem \ref{seq CM}}]
In the setting above, if $G\circ\mathcal{H}$ is multi-clique corona graph, then we have the following statements: 
\begin{enumerate}
\item The regularity of $S/I(G\circ\mathcal{H})$ is equal to 
\begin{align*}
{\rm reg}(S/I(G\circ\mathcal{H}))&= {\rm im}(G\circ\mathcal{H}) \\
&=|\{H_{i}^{j}\,\,: H_{i}^{j}\mbox{ is not the complete graph }K_{1}\}|. 
\end{align*}
\item $G\circ\mathcal{H}$ is vertex decomposable. \\
\item The projective dimension of $S/I(G\circ\mathcal{H})$ is equal to 
$${\rm pd}(S/I(G\circ\mathcal{H}))={\rm bight}I(G\circ\mathcal{H}).$$
\end{enumerate}
\end{thm}


\section{Preliminaries}
In this section, we recall several definitions and known results from graph theory and combinatorial commutative algebra that will be used throughout this paper. We refer the reader to \cite{bh, hh2, v0, v2} for the detailed information on combinatorial and algebraic background. 

\subsection{Combinatorial background.}
Let $G=(V(G), E(G))$ be a graph, where $V(G)$ denotes the vertex set and $E(G)$ denotes the edge set of $G$. For a subset $W$ of $V(G)$, we denote the induced subgraph on $W$ by $G|_{W}$. A subset $C$ of $V(G)$ is called {\it vertex cover}, if $C\cap e\neq\emptyset$ for every edge $e$ of $G$. In particular, if $C$ is a vertex cover that is minimal with respect to inclusion, then $C$ is called a {\it minimal vertex cover} of $G$. Let ${\rm Min}(G)$ be the set of minimal vertex covers of $G$. A subset $A$ of $V(G)$ is called an {\it independent set}, if no two vertices in $A$ are adjacent to each other. In particular, if $A$ is an independent set of $G$ that is maximal with respect to inclusion, then $A$ is called a {\it maximal independent set} of $G$. It is known that $${\rm ht}(I(G))=\min\{|C|\,\,: C\mbox{ is a minimal vertex cover of }G\}$$and$${\rm bight}(I(G))=\max\{|C|\,\,: C\mbox{ is a minimal vertex cover of }G\}.$$
Next, let us recall the definition of Stanley--Reisner rings. Let $n$ be a positive integer and let $[n]=\{1,\ldots, n\}$. A {\it simplicial complex }$\Delta$ on $[n]$ is a nonempty subset of the power set $2^{[n]}$ of $[n]$ such that $G\in\Delta$, if $G\subset F$ and $F\in \Delta$. An element of $\Delta$ is called face. The {\it Stanley--Reisner ideal} $I_{\Delta}$ of $\Delta$ is the ideal generated by squarefree monomials corresponding to minimal non-faces of $\Delta$, namely,$$I_{\Delta}=(x_{i_{1}}\cdots x_{i_{r}}\,\,: \{i_{1},\ldots, i_{r}\}\mbox{ is a minimal non-face of }\Delta).$$For simplicial complexes $\Delta$ and $\Delta^{\prime}$ whose vertex sets are disjoint, the {\it join} $\Delta*\Delta^{\prime}$ is the simplicial complex whose faces $F\cup F^{\prime}$, where $F\in\Delta$ and $F^{\prime}\in\Delta^{\prime}$. Let $\Delta$ be a $(d-1)$-dimensional simplicial complex. For the relationship between edge ideals and Stanley--Reisner ideals, we have the following notion called an independence complex. The {\it independence complex} of $G$ is the set of independent sets of $G$, which forms a simplicial complex $\Delta(G)$. It is known that the edge ideal of $G$ coincides with the Stanley--Reisner ideal of $\Delta(G)$ (see, for example, \cite[p. 73, Lemma 31]{v0}). For a graph $G$, a subset $M$ of $E(G)$ is called a {\it matching}, if no two  edges in $M$ share a common vertex. Then the {\it matching number} of $G$, denoted by ${\rm m}(G)$ is defined by $${\rm m}(G)=\max\{|M|\,\,: M\mbox{ is a matching of }G\}.$$ Moreover, if there is no edge in $E(G)\setminus M$ that is contained in the union of edges of $M$, then $M$ is called an {\it induced matching}. Then, the {\it induced matching number} of $G$, denoted by ${\rm im}(G)$, is defined as $${\rm im}(G)=\max\{|M|\,\,: M\mbox{ is an induced matching of }G\}.$$Moreover, let us recall the definitions of vertex decomposable and sequentially Cohen--Macaulay complexes, as defined in \cite{bw} and \cite{s}.
For a simplicial complex $\Delta$ on the vertex set $V=\{x_{1},\ldots, x_{n}\}$, $\Delta$ is called {\it vertex decomposable}, if either: 
\begin{enumerate}
\item $\Delta=\langle\{x_{1},\ldots, x_{n}\}\rangle$, or $\Delta=\emptyset$. 
\item There exists a vertex $x\in V$ such that ${\rm link}_{\Delta}(\{x\})$ and ${\rm del}_{\Delta}(\{x\})$ are vertex decomposable, and every facet of ${\rm del}_{\Delta}(\{x\})$ is a facet of $\Delta$,
\end{enumerate}
where $${\rm link}_{\Delta}(\{x\})=\{F\in\Delta\,\,: \{x\}\cap F=\emptyset\mbox{ and }\{x\}\cup F\in\Delta\}$$and$${\rm del}_{\Delta}(\{x\})=\{F\in\Delta\,\,: \{x\}\cap F=\emptyset\}.$$Moreover, for a graded module $M$ over $S=\Bbbk[x_{1},\ldots, x_{n}]$, $M$ is called {\it sequentially Cohen--Macaulay}, if there exists a filtration$$0=M_{0}\subset M_{1}\subset\cdots\subset M_{t}=M$$of $M$ by graded $S$-modules such that $\dim M_{i}/M_{i-1}<\dim M_{i+1}/M_{i}$, and $M_{i}/M_{i-1}$ is Cohen--Macaulay for all $i$. For a graph $G$, we say that $G$ is {\it vertex decomposable} , if the independence complex $\Delta(G)$ is vertex decomposable. Also, $G$ is called {\it sequentially Cohen--Macaulay} if $S/I(G)$ is sequentially Cohen--Macaulay. 

\subsection{Algebraic background}
Let us recall the definition of Serre's condition $(S_{r})$. For a ring $R$, an $R$-module $M$ and a positive integer $r$, $M$ satisfies {\it Serre's condition} $(S_{r})$ if the inequality ${\rm depth}(M_{\mathfrak{p}})\geq\min\{r,\dim M_{\mathfrak{p}}\}$ holds for every $\mathfrak{p}\in{\rm Spec}(R)$. One can easily check that $M$ is Cohen--Macaulay if and only if $M$ satisfies Serre's condition $(S_{\dim M})$. Also, let $S=\Bbbk[x_{1},\ldots, x_{n}]$ be a polynomial ring in $n$ variables over an arbitrary field $\Bbbk$ and let $M$ be a finitely generated graded $S$-module. Then, $M$ admits a {\it graded minimal free resolution} of the form $$0\rightarrow\bigoplus_{j\in\mathbb{Z}}S(-j)^{\beta_{p,j}(M)}\rightarrow\cdots\rightarrow\bigoplus_{j\in\mathbb{Z}}S(-j)^{\beta_{0,j}(M)}\rightarrow M\rightarrow 0,$$
where $S(-j)$ is the graded $S$-module with grading $S(-j)_{k}=S_{-j+k}$. The number $\beta_{i,j}(M)$ is called the $(i,j)$-th {\it graded Betti number} of $M$. The {\it Castelnuovo--Mumford regularity} of $M$ is defined by $${\rm reg}\hspace{0.05cm}M=\max\{j-i\,\,:\beta_{i,j}(M)\neq0\}.$$Also, the {\it projective dimension} of $M$ is defined by $${\rm pd}\hspace{0.05cm}M=\max\{i\,\,:\beta_{i,j}(M)\neq0\mbox{ for some }j\}.$$


\section{Algebraic aspects of rooted products of graphs}
In this section, we study rooted products of graphs from the viewpoint of combinatorial commutative algebra. In particular, we focus on the edge ideals associated with the rooted product of graphs. We provide a necessary and sufficient condition for such graphs to satisfy Serre's condition $(S_{r})$. Consequently, we obtain that a necessary and sufficient condition for such graphs to be Cohen--Macaulay. This result generalizes a celebrated theorem by Villarreal, which states that the edge ideal of a whisker graph is always Cohen–Macaulay \cite[Proposition 2.2]{v1}. Let us recall the definition of the rooted product of graphs, which was introduced by Godsil and McKay \cite[Definition 1.1]{gm}. 

\begin{defi}
Let $H_{1}, \ldots, H_{n}$ be graphs with $V(H_{i})\cap V(H_{j})=\emptyset$ for all $i\neq j$ and fix a vertex $x_{i} \in V(H_{i})$ for each $i = 1, \ldots, n$. Let $G_0$ be a graph on the vertex set $X_{[n]}=\{x_{1}, \ldots, x_{n}\}$, and set $\mathcal{H} = \{H_{1}, \ldots, H_{n}\}$.
Then, the graph $G_{0}(\mathcal{H})$, whose vertex set is$$V(G_{0}(\mathcal{H}))=\bigcup_{1\leq i\leq n}V(H_{i})$$
and whose edge set is$$E(G_{0}(\mathcal{H}))=E(G_{0})\cup\left(\bigcup_{1\leq i\leq n}E(H_{i})\right)$$is called the rooted product of $G_{0}$ by $\mathcal{H}$. 
\end{defi}

Throughout, we always assume that the following conditions hold for $G_{0}(\mathcal{H})$:
\begin{itemize}
\item[$(\ast)$] $m\geq 2$ and $H_{i}$ is connected for all $i$ and $G_{0}(\mathcal{H})$ is not the empty graph. 
\end{itemize} 

First, we give a characterization of the well-coveredness of rooted products of graphs. To this end, we introduce the following notion. 

\begin{defi}
Let $G$ be a graph on the vertex set $V$ and let $x$ be a vertex of $G$. Then, we call $G$ is {\it 2-pure with respect to $x$}, if the following conditions are satisfied: 
\begin{enumerate}
\item $G$ is pure. 
\item $G-x$ is pure with $\dim S/I(G)=\dim S/(I(G)+x)$. 
\end{enumerate}
If $G$ is 2-pure for all $x$, then $G$ is called {\it 2-pure}. 
\end{defi}

\begin{rem}
By considering the independence complex, $G$ is 2-pure if and only if $\Delta(G)$ is 2-pure in the sense of simplicial complexes. Hence, we also refer to this property as 2-pure. In graph-theoretic terms, $G$ is known to belong to the class $W_{2}$ in \cite{s1}. 

\end{rem}

The property of being 2-pure with respect to a vertex is characterized as follows.

\begin{prop}\label{2-pure}
Let $G$ be a graph and $x$ a vertex of $G$ such that $x$ is not isolated vertex. 
Then the following conditions are equivalent: 
\begin{enumerate}
\item $G$ is 2-pure with respect to $x$. 
\item $G$ is well-covered and for every maximal independent set $A$ of $G$ that contains $x$, there exists a maximal independent set $B$ of $G$  such that $B\supset A\setminus\{x\}$. 
\item $G$ is well-covered and $|C\cap V(G)|={\rm ht}I(G)$ for a minimal vertex cover $C$ of $G+xy$, where $y$ is a new vertex. 
\end{enumerate}
\end{prop}
\begin{proof}
First, we prove the equivalence between (1) and (2). Suppose that $G$ is 2-pure with respect to $x$. Fix a maximal independent set $A$ of $G$ that contains $x$. Then, by the assumption that $\dim S/I(G)=\dim S/(I(G)+x)$, $A\setminus\{x\}$ is not a maximal independent set of $G-x$. Hence, there exists a maximal independent set $B$ of $G$ such that $B\supset A\setminus\{x\}$. Conversely, we suppose that the condition (2) holds. Fix a maximal independent set $A$ of $G$ that contains $x$. To prove the well-coveredness of $G-x$, we must prove $A\setminus\{x\}$ is not a maximal independent set in $G-x$. By the condition of (2), there exists a maximal independent set $B$ of G such that $B\supset A\setminus\{x\}$. Notice that $B$ does not contain $x$ by the well-coveredness of $G$. Therefore, $B$ is a maximal independent set of $G-x$, and hence $A\setminus\{x\}$ is not a maximal independent set in $G-x$, as required. Next, we prove the equivalence between (2) and (3). By considering the complement, it is enough to show that the condition (3) holds if and only if  for every minimal vertex cover $C$ of $G$ that does not contain $x$, there exists a minimal vertex cover $D$ of $G$ that contains $x$ such that $D \subseteq C \cup \{x\}$. Notice that for every minimal vertex cover $C$ of $G+xy$, $C$ satisfies one of the following conditions:
\begin{enumerate}[label=(\alph*)]
\item $C$ is a minimal vertex cover of $G$,
\item there exists a minimal vertex cover $D$ of $G$ such that $D\cup\{y\}=C$,
\item  there exists a minimal vertex cover $D$ of $G$ such that $D\cup\{x\}=C$. 
\end{enumerate}
This shows that the condition (3) holds if and only if $C$ is not a minimal vertex cover of $G$ for all $C\in\{D\cup\{x\}\,\,: x\notin D\in{\rm Min}(G)\}$. The latter one is equivalent to the statement that for every minimal vertex cover $C$ of $G$ that does not contain $x$, there exists a minimal vertex cover $D$ of $G$ that contains $x$ such that $D \subset C\cup\{x\}$, as required. 
\end{proof}

While \cite[Proposition 3.2]{mfy} assumes that $G_{0}$ has no isolated vertices, 
we use the notion of 2-purity with respect to a vertex to characterize all graphs $G_{0}$ on the vertex set $X_{[n]}$ such that the rooted product $G_{0}(\mathcal{H})$ is well-covered.

\begin{prop}\label{well-covered}
Assume that the condition $(\ast)$ holds. Then, the following conditions are equivalent: 
\begin{enumerate}
\item $H_{i}$ is 2-pure with respect to $x_{i}$ for all $i$. 
\item $G_{0}(\mathcal{H})$ is well-covered for every graph $G_{0}$ on the vertex set $X_{[m]}$. 
\end{enumerate}
Moreover, under the assumption that $|V(H_{i})| \geq 2$ for all $i$, the following condition is equivalent to the conditions above:
\begin{enumerate}[start=3]
\item $G(\mathcal{H})$ is well-covered for every graph $G$ in $\mathcal{G}$, 
\end{enumerate}
where $\mathcal{G}$ is a family of graphs on the vertex set $X_{[m]}$ which satisfies the following condition: 
\begin{center}
For any $i=1,\ldots, m$, there exists $G\in\mathcal{G}$ such that $x_{i}$ is not an isolated vertex in $G$. 
\end{center}
\end{prop}
\begin{proof}
First, we suppose that $H_{i}$ is 2-pure with respect to $x_{i}$ for all $i$. Fix a graph $G_{0}$ on the vertex set $X_{[n]}$ and fix a minimal vertex cover $C$ of $G_{0}(\mathcal{H})$. We set $C|_{V(G_{0})}=\{x_{i}\,\,: x_{i}\in C\}$. If $C|_{V(G_{0})}=\emptyset$, then clearly$$|C|=\displaystyle\sum_{1\leq i\leq n}{\rm ht}I(H_{i}).$$Hence we may assume that $C|_{V(G_{0})}\neq\emptyset$. Without loss of generality, we may assume that $C|_{V(G_0)} = \{x_1, \ldots, x_s\}$ for some $1 \leq s \leq n$, after relabeling the vertices of $G_0$. Then, for all $i=s+1,\ldots, n$, by Proposition \ref{2-pure} and well-coveredness of $H_{i}$, there exists a minimal vertex cover $C_{i}^{\prime}$ of $H_{i}$ such that $x_{i}\in C_{i}\subset C|_{V(H_{i})}\cup\{x_{i}\}$ and $|C_{i}^{\prime}|=|C|_{V(H_{i})}|$. By replacing $C|_{V(H_{i})}$ by $C^{\prime}$, we get a minimal vertex cover $C^{\prime}$ of $G_{0}(\mathcal{H})$, namely, we set $$C^{\prime} = \left( \bigcup_{1 \leq i \leq s} C|_{V(H_{i})} \right) \cup \left( \bigcup_{s+1 \leq i \leq n} C_{i}^{\prime} \right).$$
Then, $|C|=|C^{\prime}|$. By applying this operation to every minimal vertex cover of $G_{0}(\mathcal{H})$ with $C|_{V(G_{0})}\neq\emptyset$, we can assume, while preserving their cardinalities, that each minimal vertex cover of $G_{0}(\mathcal{H})$ contains all vertices in $G_{0}$. Hence, we have $$|C|=n+\displaystyle\sum_{1\leq i\leq n}({\rm ht}I(H_{i})-1)=\displaystyle\sum_{1\leq i\leq n}{\rm ht}I(H_{i})$$for any minimal vertex cover $C$ of $G_{0}(\mathcal{H})$. Therefore, $G_{0}(\mathcal{H})$ is well-covered. We now prove the converse implication. To this end, we suppose that $G_{0}(\mathcal{H})$ is well-covered for every graph $G_{0}$ on the vertex set $X_{[m]}$. By considering a graph $G_{0}$ as the empty graph, one can see that $H_{i}$ is well-covered for all $i$. Hence it suffices to show that $H_{i}$ is 2-pure with respect to $x_{i}$ for all $i$. Suppose not, that is, we suppose that there exists $i$ such that $H_{i}$ is not 2-pure with respect to $x_{i}$ for all $i$. If there exist distinct $p, q$ such that $H_{p}$ and $H_{q}$ are single vertex graph, then, by considering $G_{0}$ as the graph such that $\{x_{p}, x_{r}\}$, and $\{x_{q}, x_{r}\}$ are only edges for some $r$ such that $H_{r}$ is not a single vertex graph, we see that $G_{0}(\mathcal{H})$ is not well-covered, by considering minimal vertex covers $C_{r}, C_{r}^{\prime}$ of $H_{r}$ such that $x_{r}\in C_{r}$ and $x_{r}\notin C_{r}^{\prime}$. By the assumption that $(\ast)$, there exists $j$ such that $j\neq i$ and $H_{j}$ is not a single vertex graph. We consider a graph $G_{0}$ such that $\{x_{i}, x_{j}\}$ is the only edge of $G_{0}$. Then, since $H_{i}$ is not 2-pure with respect to $x_{i}$, by Proposition \ref{2-pure} (c), there exists a minimal vertex cover $C_{i}$ of $H_{i}$ such that $C_{i}\cup\{x_{j}\}$ is also a minimal vertex cover of $H_{i}+x_{j}$. Also, fix a minimal vertex cover $C_{j}$ of $H_{j}$ such that $x_{j}\notin C_{j}$ and a minimal vertex cover $C_{j}^{\prime}$ of $H_{j}$ such that $x_{j}\in C_{j}^{\prime}$ since $|V(H_{j})|\geq2$. Then, $C_{i}\cup C_{j}^{\prime}$ is a minimal vertex cover of  $(G_{0}(\mathcal{H}))|_{V(H_{i})\cup V(H_{j})}$, where this graph is the induced subgraph of $G_{0}(\mathcal{H})$ on $V(H_{i})\cup V(H_{j})$. Now, since $H_{i}$ is not 2-pure with respect to $x_{i}$, by  Proposition \ref{2-pure}, $C_{i}\cup\{x_{i}\}$ is a minimal vertex cover of $H_{i}+x_{j}$. Hence, it follows that $(C_{i}\cup\{x_{i}\})\cup C_{j}$ is a minimal vertex cover of $G_{0}(\mathcal{H})|_{V(H_{i})\cup V(H_{j})}$. Therefore, since $|C_{j}|=|C_{j}^{\prime}|$, we obtain that $|C_{i}\cup C_{j}^{\prime}|<|(C_{i}\cup\{x_{i}\})\cup C_{j}|$, which contradicts the assumption that $G_{0}(\mathcal{H})$ is well-covered, as required. 
\end{proof}

In \cite[Proposition 2.3]{mfy}, vertex decomposability of rooted products of graphs is investigated. We refine this result in pure case, by showing that, when the base graph $G_{0}$ varies, the conditions given in \cite[Proposition 2.3]{mfy} are in fact equivalent. 

\begin{prop}\label{vertex decomp}
Assume that the condition $(\ast)$ holds. Then the following conditions are equivalent: 
\begin{enumerate}
\item $H_{i}$ is vertex decomposable such that $x_{i}$ is a shedding vertex and $H_{i}$ is 2-pure with respect to $x_{i}$ for all $i$. 
\item $G_{0}(\mathcal{H})$ is pure vertex decomposable for every graph $G_{0}$ on the vertex set $X_{[m]}$.
\end{enumerate}
Moreover, under the assumption that $|V(H_{i})| \geq 2$ for all $i$, the following condition is equivalent to the conditions above:
\begin{enumerate}[start=3]
\item $G(\mathcal{H})$ is vertex decomposable for every graph $G$ in $\mathcal{G}$, 
\end{enumerate}
where $\mathcal{G}$ is a family of graphs on the vertex set $X_{[m]}$ which satisfies the following condition: 
\begin{center}
For any $i=1,\ldots, m$, there exists $G\in\mathcal{G}$ such that $x_{i}$ is not an isolated vertex in $G$. 
\end{center}
\end{prop}
\begin{proof}
First, we suppose that $H_{i}$ is vertex decomposable such that $x_{i}$ is a shedding vertex and $H_{i}$ is 2-pure with respect to $x_{i}$ for all $i$. Fix a graph $G_{0}$ on the vertex set $X_{[m]}$. We proceed by induction on $|V(G_{0})|=m\geq1$. In the case of $|V(G_{0})| = 1$, the graph $G_{0}(\mathcal{H})$ is simply $H_1$, which is clearly vertex decomposable. Now assume that $|V(G_{0})| > 1$. Then we have $${\rm del}_{\Delta(G_{0}(\mathcal{H}))}(x_{1})=\Delta(G_{0}(\mathcal{H})-x_{1})=\Delta((G_{0}-x_{1})(\mathcal{H}\setminus H_{1}))*\Delta(H_{1}-x_{1}).$$Since these simplicial complexes are vertex decomposable and the join of vertex decomposable simplicial complexes are also vertex decomposable, ${\rm del}_{\Delta(G_{0}(\mathcal{H}))}$ is vertex decomposable. Moreover, we prove that ${\rm link}_{\Delta(G_{0}(\mathcal{H}))}(x_{1})$ is vertex decomposable. Let $N_{G_{0}}(x_{1})=\{x_{i_{1}},\ldots, x_{i_{s}}\}$ and let $\Gamma=\Delta(H_{i_{1}}-x_{i_{1}})*\cdots*\Delta(H_{i_{s}}-x_{i_{s}})$. Then we have 
\begin{align*}
{\rm link}_{\Delta(G_{0}(\mathcal{H}))}(x_{1})&=\Delta(G_{0}-N[x_{1}]) \\
&= \Delta((G_{0}-N_{G_{0}}[x_{1}])(\mathcal{H}\setminus\{H_{i_{1}},\ldots, H_{i_{s}}\}))*\Delta(H_{1}-N_{H_{1}}[x_{1}])*\Gamma.
\end{align*}
By the induction hypothesis and the assumption, ${\rm link}_{\Delta(G_{0}(\mathcal{H}))}(x_{1})$ is vertex decomposable, as required. 
Conversely, we suppose that $G_{0}(\mathcal{H})$ is pure vertex decomposable for every graph $G_{0}$ on the vertex set $X_{[m]}$. By considering a graph $G_{0}$ as the empty graph, we see that $H_{i}$ is vertex decomposable for any $i$ by \cite[Proposition 2.4]{pb}. Also, by \cite[Proposition 2.3]{pb}, the link of any face of vertex decomposable simplicial complex, is also vertex decomposable. Hence $H_{i}-N[x_{i}]$ is a vertex decomposable for any $i$. To prove that $x_{i}$ is a shedding vertex, we must check that $H_{i}-x_{i}$ is vertex decomposable.  We consider a graph $G_{0}$. Fix $i$ and $j$. Consider a graph $G_{0}$ in which $x_{i}$ and $x_{i}$ are the only adjacent vertices; all other vertices are isolated. Then, since $\Delta(G_{0}(\mathcal{H})-N[x_{j}])$ is vertex decomposable, we see that $\Delta(H_{i}-x_{i})$ is vertex decomposable by \cite[Proposition 2.4]{pb}, which completes the proof.

It remains to prove the second assertion of this theorem. $(2)\Rightarrow(3)$ is clear. Hence, it suffices to prove that $(3)\Rightarrow (1)$. By Proposition \ref{well-covered}, we have $H_{i}$ is 2-pure with respect to $x_{i}$ for all $i$.  By symmetry, it suffices to prove the assertions for $i=1$. Take $G_{1}\in\mathcal{G}$ be a graph such that $x_{1}$ is not an isolated vertex in $G_{1}$. For every $j\geq2$, let $A_{j}$ be a facet of $\Delta(H_{j})$ with $x_{j}\notin A_{j}$. Then we have $A=\cup_{j\geq2}A_{j}$ is a face of $\Delta(G_{1}(\mathcal{H}))$. Since $\Delta(G_{1}(\mathcal{H}))$ is vertex decomposable, ${\rm link}_{\Delta(G_{1}(\mathcal{H}))}A=\Delta(H_{1})$ is also vertex decomposable. Hence $H_{1}$ is vertex decomposable. To prove that $x_{1}$ is a shedding vertex of $H_{1}$, we must check that $H_{1}-x_{1}$ is vertex decomposable since $H_{1}-N[x_{1}]$ is vertex decomposable.  Since $x_{1}$ is not an isolated vertex in $G_{1}$, we can take $i$ such that $\{x_{1}, x_{i}\}$ is an edge of $G_{1}$. Then, we see that $H_{1}-x_{1}$ is a connected component of $G_{1}(\mathcal{H})-N[x_{i}]$. Since ${\rm link}_{\Delta(G_{1}(\mathcal{H}))}x_{i}$ is vertex decomposable, by \cite[Proposition 2.4]{pb}, it follows that $H_{1}-x_{1}$ is vertex decomposable, as required. 
\end{proof}

We prove the main theorem in this section that characterizes when $S/I(G_{0}(\mathcal{H}))$ satisfies Serre's condition $(S_{r})$. To this end, we introduce the following notion: 

\begin{defi}
Let $G$ be a graph, $x$ be a vertex of $G$ and let $r\geq2$. Then, we call $G$ satisfies {\it 2-Serre's condition $(S_{r})$ with respect to $x$}, if the following conditions are satisfied: 
\begin{enumerate}
\item $G$ satisfies Serre's condition $(S_{r})$. 
\item $G-x$ satisfies Serre's condition $(S_{r})$ with $\dim S/I(G)=\dim S/(I(G)+x)$. 
\end{enumerate}
If $G$ satisfies 2-Serre's condition $(S_{r})$ for all vertices, then we call $G$ satisfies {\it 2-Serre's condition $(S_{r})$}.  
\end{defi}

As a generalization of \cite[Lemma 2.5]{hp}, we prove the following lemma: 

\begin{lemma}\label{lemma1}
Let $G$ be a graph, $x$ be a vertex of $G$ and $r\geq2$. If $G-x$ and $G-N[x]$ satisfy Serre's condition $(S_{r})$ with $\alpha(G-x)=\alpha(G-N[x])+1$, then $G$ satisfies Serre's condition $(S_{r})$. 
\end{lemma}
\begin{proof}
Let $\Delta_{1}=\Delta(G-x)$ and $\Delta_{2}=\Delta((G-N[x])\cup\{x\})$. Notice that $\Delta(G)=\Delta_{1}\cup\Delta_{2}$ and $\Delta_{1}\cap\Delta_{2}=\Delta(G-N[x])$. Also, from \cite[Lemma 2.6]{mt}, $G-x$ and $G-N[x]$ is well-covered. Hence, by $\alpha(G-x)=\alpha(G-x)+1$, we see that $G$ is well-covered and $\alpha(G)=\alpha(G-x)$. Indeed, let $A$ be a maximal independent set of $G$. If $x\in A$, then
$A\setminus\{x\}$ is a maximal independent set of $G-N[x]$ and hence $|A| \;=\; 1+\alpha(G-N[x]) \;=\; \alpha(G-x)$. 
Otherwise, that is, if $x\notin A$, then $A$ is a maximal independent set of $G-x$; hence
$|A| \;=\; \alpha(G-x)$. To prove the statement of this lemma, by \cite[page 4, following Theorem 1.7]{t1}, fix $r\geq2$ and it suffices to prove that $\widetilde{H}_{i}({\rm link}_{\Delta(G)}F;\Bbbk)=0$ for all $i\leq r-2$ and all $F\in\Delta(G)$ with $|F|\leq\dim S/I(G)-i-2$. Fix $i\leq r-2$ and $F\in\Delta(G)$ with  $|F|\leq\dim S/I(G)-i-2$. Since $\alpha(G)=\alpha(G-x)$, we have $|F|\leq\dim S/(I(G)+(x))-r-2$. If $F\in\Delta_{1}$ and $F\notin\Delta_{2}$, then since $|F|\leq\dim S/I(G)-i-2=\dim S/(I(G)+(x))-i-2$, we obtain that $\widetilde{H}_{i}({\rm link}_{\Delta}F;\Bbbk)=\widetilde{H}_{i}({\rm link}_{\Delta_{1}}F;\Bbbk)=0$. Also, if $F\in\Delta_{2}$ and $F\notin\Delta_{1}$, then, since $\alpha(G)=\alpha(G-N[x])+1=\alpha(G-N(x))$, we have $\widetilde{H}_{i}({\rm link}_{\Delta(G)}F;\Bbbk)=\widetilde{H}_{i}({\rm link}_{\Delta_{2}}F;\Bbbk)=0$. Suppose that $F\in\Delta_{1}\cap\Delta_{2}$. Then, from the Mayer--Vietoris exact sequence, we obtain that 
$$\widetilde{H}_{i}({\rm link}_{\Delta_{1}}F;\Bbbk)\oplus\widetilde{H}_{i}({\rm link}_{\Delta_{2}}F;\Bbbk)\rightarrow\widetilde{H}_{i}({\rm link}_{\Delta(G)}F;\Bbbk)\rightarrow\widetilde{H}_{i-1}({\rm link}_{\Delta_{1}\cap\Delta_{2}}F;\Bbbk).$$Now, since $x\notin F$ and $\Delta_{2}$ is a cone, ${\rm link}_{\Delta_{2}}F$ is also a cone, and acyclic. Also, since $\Delta_{p}$ satisfies Serre's condition $(S_{r})$, we see that $\widetilde{H}_{i}({\rm link}_{\Delta_{1}}F;\Bbbk)=0$. Moreover, since $|F|\leq\dim S/I(G)-i-2=\dim \Bbbk[V(G-N[x])]/I(G-N[x])-(i-1)-2$ and $\Delta(G-N[x])$ satisfies Serre's condition $(S_{r})$, we have $$\widetilde{H}_{i-1}({\rm link}_{\Delta_{1}\cap\Delta_{2}}F;\Bbbk)=\widetilde{H}_{i-1}({\rm link}_{\Delta(G-N[x])}F;\Bbbk)=0.$$From the exact sequence above, we obtain that $\widetilde{H}_{i}({\rm link}_{\Delta(G)}F;\Bbbk)=0$, as required, 
\end{proof}

From Lemma \ref{lemma1}, we obtain the following statement:  

\begin{prop}\label{2-S_{2}}
For a graph $G$, a vertex $x$ of $G$ and $r\geq2$, $G$ satisfies 2-Serre's condition $(S_{r})$ with respect to $x$ if and only if $G-x$ and $G-N[x]$ satisfy Serre's condition $(S_{r})$ with $\alpha(G-x)=\alpha(G-N[x])+1$. 
\end{prop}
\begin{proof}
If $G-x$ and $G-N[x]$ satisfy Serre's condition $(S_{r})$ with $\alpha(G-x)=\alpha(G-N[x])+1$, then, by Lemma \ref{lemma1}, we see that $G$ satisfies 2-Serre's condition $(S_{r})$ with respect to $x$. Suppose that $G$ satisfies 2-Serre's condition $(S_{r})$ with respect to $x$. Since $\Delta(G-N[x])={\rm link}_{\Delta(G)}(x)$ and the fact that the link of any face of simplicial complex satisfies Serre's condition $(S_{r})$ also satisfies Serre's condition $(S_{r})$ by \cite[Lemma 2.2]{htyn}, it suffices to check that $\alpha(G-x)=\alpha(G-N[x])+1$. Since $\alpha(G)=\alpha(G-x)$, we see that $\alpha(G-x)\geq\alpha(G-N[x])+1$. Fix a facet $A$ of $\Delta(G)$ with $x\in A$. The, since $A\setminus\{x\}\in\Delta(G)$, we obtain that $\alpha(G)-1=|A\setminus\{x\}|\leq\alpha(G-N[x])$, as required. 
\end{proof}

Also, we have the following lemma by \cite[Theorem 6 (b)]{ty}:  
\begin{lemma}\label{lemma2}
For simplicial complexes $\Delta_{1}$ and $\Delta_{2}$ whose vertex sets are disjoint and $r\geq2$, $\Delta_{1}*\Delta_{2}$ satisfies Serre's condition $(S_{r})$ if and only if $\Delta_{1}$ and $\Delta_{2}$ satisfy Serre's condition $(S_{r})$. 
\end{lemma}

We now prove the main theorem in this section that is a characterization of when the edge ideals of rooted products of graphs satisfy Serre's condition $(S_{r})$ for $r\geq2$. 

\begin{thm}\label{Serre}
Assume that the condition $(\ast)$ holds and let $r\geq2$. Then, the following conditions are equivalent: 
\begin{enumerate}
\item $H_{i}$ satisfies 2-Serre's condition $(S_{r})$ with respect to $x_{i}$ for all $i$. 
\item $G_{0}(\mathcal{H})$ satisfies Serre's condition $(S_{r})$ for every graph $G_{0}$ on the vertex set $X_{[m]}$. 
\end{enumerate}
Moreover, under the assumption that $|V(H_{i})| \geq 2$ for all $i$, the following condition is equivalent to the conditions above:
\begin{enumerate}[start=3]
\item $G(\mathcal{H})$ satisfies Serre's condition $(S_{r})$ for every graph $G$ in $\mathcal{G}$, 
\end{enumerate}
where $\mathcal{G}$ is a family of graphs on the vertex set $X_{[m]}$ which satisfies the following condition: 

\begin{center}
For any $i=1,\ldots, m$, there exists $G\in\mathcal{G}$ such that $x_{i}$ is not an isolated vertex in $G$. 
\end{center}
\end{thm}
\begin{proof}
First, we suppose that $H_{i}$ satisfies 2-Serre's condition $(S_{r})$ with respect to $x_{i}$ for all $i$. We prove by induction on the number of vertices of $G_{0}$. If $m=1$, then the assertion clearly holds. Hence we suppose that $m>1$. We set $N_{G_{0}}(x_{1})=\{x_{i_{1}},\ldots, x_{i_{s}}\}$ and set $V(G_{0})\setminus N_{G_{0}}[x_{1}]=\{x_{i_{s+1}},\ldots, x_{i_{t}}\}$. It is enough to show that $G_{0}(\mathcal{H})-x_{1}$ and $G_{0}(\mathcal{H})-N[x_{1}]$ satisfy Serre's condition $(S_{r})$ with $\alpha(G_{0}(\mathcal{H})-x_{1})=\alpha(G_{0}(\mathcal{H})-N[x_{1}])+1,$ by Lemma \ref{lemma1}. Now we have $$\Delta(G_{0}(\mathcal{H})-x_{1})=\Delta(H_{1}-x_{1})*\Delta((G_{0}-x_{1})(\mathcal{H}\setminus\{H_{1}\})).$$Hence, since $|V(G_{0}-x_{1})|<n$, by the induction hypothesis, $\Delta((G_{0}-x_{1})(\mathcal{H}\setminus\{H_{1}\}))$ satisfies Serre's condition $(S_{r})$ and thus, we see that $\Delta(G_{0}(\mathcal{H})-x_{1})$ satisfies Serre's condition $(S_{r})$ by the assumption that $H_{1}-x_{1}$ satisfies Serre's condition $(S_{r})$. We set $\Gamma=\Delta(H_{i_{1}}-x_{i_{1}})*\cdots*\Delta(H_{i_{s}}-x_{i_{s}})$ and $\Gamma^{\prime}=\Delta(H_{i_{s+1}})*\cdots*\Delta(H_{i_{t}})$. Then, we have $$\Delta(G_{0}(\mathcal{H})-N[x_{1}])=\Delta(H_{1}-N_{H_{1}}[x_{1}])*\Gamma*\Gamma^{\prime}.$$Notice that $\Delta(H_{1} - N_{H_{1}}[x_{1}]) = {\rm link}_{\Delta(H_{1})}(x_{1})$. Since $\Delta(H_{1})$ satisfies Serre's condition $(S_{r})$, and hence the link of any face satisfies Serre's condition $(S_{r})$ by \cite[Lemma 2.2]{htyn}, it follows that $\Delta(H_{1} - N_{H_{1}}[x_{1}])$ satisfies Serre's condition $(S_{r})$. By the assumption that $H_{k}$ and $H_{i_{j}}-x_{i_{j}}$ satisfy Serre's condition $(S_{r})$ for all $j$ and $k$, $\Delta(G_{0}(\mathcal{H})-N[x_{1}])$ satisfies Serre's condition $(S_{r})$. Finally, we must check that $\alpha(G_{0}(\mathcal{H})-x_{1})=\alpha(G_{0}(\mathcal{H})-N[x_{1}])+1$. Since, $H_{1}$ is 2-pure with respect to $x_{1}$, by Proposition \ref{2-pure}, we see that $\alpha(H_{1})=\alpha(H_{1}-x_{1})$. Moreover, since $H_{i}$ is 2-pure with respect to $x_{i}$ for all $i$, by Proposition \ref{2-pure}, we have $\alpha((G_{0}-x_{1})(\mathcal{H}\setminus\{H_{1}\}))=\sum_{2\leq i\leq m}\alpha(H_{i}).$ Therefore, we obtain that $$\alpha(G_{0}(\mathcal{H})-x_{1}))=\alpha(H_{1}-x_{1})+\alpha((G_{0}-x_{1})(\mathcal{H}\setminus\{H_{1}\}))=\displaystyle\sum_{1\leq i\leq m}\alpha(H_{i}).$$On the other hand, since $H_{i}$ is 2-pure for all $i$, we have $\alpha(H_{i_{j}}-x_{i_{j}})=\alpha(H_{i_{j}})$ for all $1\leq j\leq s$ and we have$$\alpha((G_{0}-N_{G_{0}}[x_{1}])*(\mathcal{H}\setminus\{H_{1}, H_{i_{1}},\ldots, H_{i_{s}}\}))=\displaystyle\sum_{s+1\leq k\leq t}\alpha(H_{i_{k}}).$$Moreover, since there exists a maximal independent set of $H_{1}$ that contains $x_{1}$ and $\Delta(H_{1})$ is pure, we obtain that $\dim\Bbbk[{\rm link}_{\Delta(H_{1})}(x_{1})]=\dim\Bbbk[\Delta(H_{1})]-1$, that is, $\alpha(H_{1}-N_{H_{1}}[x_{1}])=\alpha(H_{1})-1$. Therefore, we obtain that 
$$\alpha(G_{0}(\mathcal{H})-N[x_{1}])=\alpha(H_{1})-1+\displaystyle\sum_{2\leq i\leq m}\alpha(H_{i}),$$which implies the desired equality. Conversely, we suppose that $G_{0}(\mathcal{H})$ satisfies Serre's condition $(S_{r})$ for every graph $G_{0}$ on the vertex set $X_{[m]}$. We prove that $H_{i}$ satisfies 2-Serre's condition $(S_{r})$ with respect to $x_{i}$ for all $i$. By considering a graph $G_{0}$ as the empty graph, one can see that $H_{i}$ satisfies Serre's condition $(S_{r})$ for all $i$ by Lemma \ref{lemma2}. Fix $i$ and $j$. Consider a graph $G_{0}$ in which $x_{i}$ and $x_{i}$ are the only adjacent vertices; all other vertices are isolated. Then, since $\Delta(G_{0}(\mathcal{H})-N[x_{j}])$ satisfies Serre's condition $(S_{r})$, we see that $\Delta(H_{i}-x_{i})$ satisfies Serre's condition $(S_{r})$ by Lemma \ref{lemma2}. Moreover, since $G_{0}(\mathcal{H})$ is well-covered, by Proposition \ref{well-covered}, $H_{i}$ is 2-pure with respect to $x_{i}$, and thus we have $\dim \Bbbk[V(H_{i})]/I(H_{i})=\dim \Bbbk[V(H_{i})]/(I(H_{i})+x_{i})$, as required. 

It remains to prove the second assertion of this theorem. Suppose that $|V(H_{i})|\geq 2$ for all $i$. It suffices to check that $G(\mathcal{H})$ satisfies Serre's condition $(S_{r})$ for every graph $G$ in $\mathcal{G}$ implies that $H_{i}$ satisfies 2-Serre's condition $(S_{r})$ with respect to $x_{i}$ for all $i$. We must check that $H_{i}$ satisfies Serre's condition $(S_{r})$ and $H_{i}$ satisfies 2-Serre's condition $(S_{r})$ with respect to $x_{i}$ for all $i$. By symmetry, it suffices to prove the assertions for $i=1$. Take $G_{1}\in\mathcal{G}$ be a graph such that $x_{1}$ is not an isolated vertex in $G_{1}$. For every $j\geq2$, let $A_{j}$ be a facet of $\Delta(H_{j})$ with $x_{j}\notin A_{j}$. Then we have $A=\cup_{j\geq2}A_{j}$ is a face of $\Delta(G_{1}(\mathcal{H}))$. Since $\Delta(G_{1}(\mathcal{H}))$ satisfies Serre's condition $(S_{r})$, ${\rm link}_{\Delta(G_{1}(\mathcal{H}))}A=\Delta(H_{1})$ also satisfies Serre's condition $(S_{r})$. Hence $H_{1}$ satisfies Serre's condition $(S_{r})$. Finally, we prove that $H_{1}-x_{1}$ satisfies Serre's condition $(S_{r})$. Since $x_{1}$ is not an  isolated vertex in $G_{1}$, we can take $i$ such that $\{x_{1}, x_{i}\}$ is an edge of $G_{1}$. Then, we see that $H_{1}-x_{1}$ is a connected component of $G_{1}(\mathcal{H})-N[x_{i}]$. Since ${\rm link}_{\Delta(G_{1}(\mathcal{H}))}x_{i}$ satisfies Serre's condition $(S_{r})$, by Lemma \ref{lemma2}, it follows that $H_{1}-x_{1}$ satisfies Serre's condition $(S_{r})$, as required. 
\end{proof}

For a fixed base graph $G_{0}$, we have the following corollaries for $G_{0}(\mathcal H)$.

\begin{cor}\label{cor about Serre}
Assume that the condition $(\ast)$ holds and let $r\geq2$. Suppose that $|V(H_{i})|\geq2$ for all $i$. Fix a graph $G_{0}$ on the vertex set $X_{[m]}$ without isolated vertices. Then, the following conditions are equivalent: 
\begin{enumerate}
\item $H_{i}$ satisfies 2-Serre's condition $(S_{r})$ with respect to $x_{i}$ for all $i$. 
\item $G_{0}(\mathcal{H})$ satisfies Serre's condition $(S_{r})$. 
\end{enumerate}
\end{cor}

\begin{cor}
Assume that the condition $(\ast)$ holds and let $r\geq2$ and suppose that $H_{i}$ is 2-pure with respect to $x_{i}$ for all $i$. Fix a graph $G_{0}$ on the vertex set $X_{[m]}$. Then, $H_{i}$ satisfies 2-Serre's condition $(S_{r})$ with respect to $x_{i}$ for all $i$ if and only if $G_{0}(\mathcal{H})$ satisfies 2-Serre's condition $(S_{r})$ with respect to $x_{i}$ for all $i$. 
\end{cor}
\begin{proof}
If $H_{i}$ satisfies 2-Serre's condition $(S_{r})$ with respect to $x_{i}$ for all $i$, then, by Proposition \ref{2-pure}, Lemma \ref{lemma2} and Theorem \ref{Serre}, we see that $G_{0}(\mathcal{H})$ satisfies 2-Serre's condition $(S_{r})$ with respect to $x_{i}$ for all $i$.  Hence we suppose that $G_{0}(\mathcal{H})$ satisfies 2-Serre's condition $(S_{r})$ with respect to $x_{i}$ for all $i$. Fix $i=1,\ldots,m$. By Lemma \ref{lemma2}, one can see that $H_{i}-x_{i}$ and $H_{i}-N[x_{i}]$ satisfy Serre's condition $(S_{r})$. By the assumption, we have$$\displaystyle\sum_{1\leq j\leq m}\alpha(H_{j})=\alpha(G_{0}(\mathcal{H}))=\alpha(G_{0}(\mathcal{H})-x_{i})=\alpha(H_{i}-x_{i})+\displaystyle\sum_{j\neq i}\alpha(H_{j}).$$This implies that $\alpha(H_{i})=\alpha(H_{i}-x_{i})$ and hence, since $H_{i}$ is well-covered, we obtain that $\alpha(H_{i}-x_{i})=\alpha(H_{i}-N_{H_{i}}[x_{i}])+1$. Therefore, by Proposition \ref{2-S_{2}}, the assertion follows.  
\end{proof}

Let us recall the definition the Serre index, which was introduced in \cite{ppty}. 
Let $\Delta$ be a simplicial complex with $\dim\Delta=d-1$. Then the {\it Serre index} of $\Bbbk[\Delta]$, denoted by {\rm S-ind}$(\Bbbk[\Delta])$, is defined by $$\mbox{{\rm S-ind}}(\Bbbk[\Delta])=\max\{r\leq d\,\,: \Bbbk[\Delta]\mbox{ satisfies Serre's condition }(S_{r})\}.$$

\begin{cor}
Assume that the condition $(\ast)$ holds and let $r\geq2$. Then we have
\begin{align*}
&\min\{\mbox{\rm S-ind}(S/I(G_{0}(\mathcal{H})))\,\,: G_{0}\mbox{ is a graph on }X_{[m]}\} \\
&=\min_{1\leq i\leq m}\{\max\{r\,\,: H_{i}\mbox{ satisfies }2\mbox{-Serre's condition }(S_{r})\mbox{ with respect to }x_{i}\}\}
\end{align*}
\end{cor}
\begin{proof}
Set $r=\min\{\mbox{\rm S-ind}(S/I(G_{0}(\mathcal{H})))\,\,: G_{0}\mbox{ is a graph on }X_{[m]}\}$ and $$p=\min_{1\leq i\leq m}\{\max\{r\,\,: H_{i}\mbox{ satisfies }2\mbox{-Serre's condition }(S_{r})\mbox{ with respect to }x_{i}\}\}.$$ Then, we see that $G_{0}(\mathcal{H})$ satisfies Serre's condition $(S_{r})$ for any graph $G_{0}$ on the vertex set $X_{[m]}$. Hence, by Theorem \ref{Serre}, $H_{i}$ satisfies 2-Serre's condition $(S_{r})$ with respect to $x_{i}$ for all $i$. Therefore, we obtain that $r\geq p$. Conversely, since $H_{i}$ satisfies 2-Serre's condition $(S_{p})$ with respect to $x_{i}$ for all $i$, by Theorem \ref{Serre}, $G_{0}(\mathcal{H})$ satisfies Serre's condition $(S_{p})$ for any graph $G_{0}$ on the vertex set $X_{[m]}$. Hence we obtain that $r\leq p$, as required. 
\end{proof}

In particular, we state the definition of 2-Cohen--Macaulay with respect to a vertex. 

\begin{defi}
Let $G$ be a graph and let $x$ be a vertex of $G$. Then, we call $G$ is {\it 2-Cohen--Macaulay with respect to $x$}, if the following conditions are satisfied: 
\begin{enumerate}
\item $G$ is Cohen--Macaulay. 
\item $G-x$ is Cohen--Macaulay with $\dim S/I(G)=\dim S/(I(G)+x)$. 
\end{enumerate}
If $G$ is 2-Cohen--Macaulay for all vertices, then $G$ is called  {\it 2-Cohen--Macaulay.} 
\end{defi}

Thanks to Theorem \ref{Serre}, in particular, we obtain that the following corollaries: 
\begin{cor}\label{CM}
Assume that the condition $(\ast)$ holds and let $r\geq2$. Then, the following conditions are equivalent: 
\begin{enumerate}
\item $H_{i}$ is 2-Cohen--Macaulay with respect to $x_{i}$ for all $i$. 
\item $G_{0}(\mathcal{H})$ is Cohen--Macaulay for every graph $G_{0}$ on the vertex set $X_{[m]}$. 
\end{enumerate}
Moreover, under the assumption that $|V(H_{i})| \geq 2$ for all $i$, the following condition is equivalent to the conditions above:
\begin{enumerate}[start=3]
\item $G(\mathcal{H})$ is Cohen--Macaulay for every graph $G$ in $\mathcal{G}$, 
\end{enumerate}
where $\mathcal{G}$ is a family of graphs on the vertex set $X_{[m]}$ which satisfies the following condition: 
\begin{center}
For any $i=1,\ldots, m$, there exists $G\in\mathcal{G}$ such that $x_{i}$ is not an isolated vertex in $G$. 
\end{center}
\end{cor}

\begin{cor}\label{cor about CM}
Assume that the condition $(\ast)$ holds and let $r\geq2$. Suppose that $|V(H_{i})|\geq2$ for all $i$. Fix a graph $G_{0}$ on the vertex set $X_{[m]}$ without isolated vertices. Then, the following conditions are equivalent: 
\begin{enumerate}
\item $H_{i}$ is 2-Cohen--Macaulay with respect to $x_{i}$ for all $i$. 
\item$G_{0}(\mathcal{H})$ is Cohen--Macaulay. 
\end{enumerate}
\end{cor}

As a corollary, we obtain the following theorem proved by Hoang and Pham. 
\begin{cor}\cite[Theorem 2.6]{hp}
Let $G$ be a graph and $\mathcal{H}=\{H_{v}\,\,: v\in V(G)\}$. The following conditions are equivalent: 
\begin{enumerate}
\item $G\circ\mathcal{H}$ is Cohen--Macaulay. 
\item $G\circ\mathcal{H}$ is well-covered. 
\item $G\circ\mathcal{H}$ is a clique corona graph. 
\end{enumerate}
\end{cor}


\section{A classification of 2-pureness and shedding vertices with $|V(G)|\leq 6$}

In this section, we define the concept of 2-pure with respect to a vertex.  We also classify when $G$ is 2-pure with respect to a vertex, in the case where $G$ has at most 6 vertices. Moreover, we classify its shedding vertices. This classification is based on the list of well-covered graphs with up to 6 vertices provided in the Appendix of \cite[pp. 635--637]{v2}. First, we define the concept of 2-pure with respect to a vertex. 

We set $U(G)=\{x\in V(G)\,\,: G\mbox{ is 2-pure with respect to }x\}$ and ${\rm Shed}(G)=\{x\in V(G)\,\,: x\mbox{ is a shedding vertex of }G\}$, where if $G$ is not vertex decomposable, then we set ${\rm Shed}(G)=\emptyset$. The following table provides a classification of all well-covered graphs $G$ with $|V(G)| \leq 6$, listing the set $U(G)$ and the set ${\rm Shed}(G)$:   

\begin{center}
\begin{longtable}{|>{\centering\arraybackslash}m{1.5cm}|c|>{\centering\arraybackslash}m{3cm}|>{\centering\arraybackslash}m{3cm}|}

\hline
$|V(G)|$ & Graph $G$ & $U(G)$ & ${\rm Shed}(G)$\\
\hline
\endfirsthead

\hline
$|V(G)|$ & Graph $G$ & $U(G)$ & ${\rm Shed}(G)$\\
\hline
\endhead

\raisebox{0pt}[0.8cm][0.8cm]{2} 
& \begin{tikzpicture}[baseline=-0.5ex, scale=1.0]
  \node[draw, circle, inner sep=1.2pt, fill=black] (a) at (0,0) {};
  \node[draw, circle, inner sep=1.2pt, fill=black] (b) at (1.2,0) {};
  \draw (a) -- (b);
  \node[below=2pt of a] {1};
  \node[below=2pt of b] {2};
\end{tikzpicture}
& \raisebox{0pt}[0.8cm][0.8cm]{$\{1, 2\}$}
& \raisebox{0pt}[0.8cm][0.8cm]{$\{1, 2\}$} \\
\hline

\raisebox{0.3cm}[0.8cm][0.8cm]{3} 
& \begin{tikzpicture}[baseline=-0.5ex, scale=1.0]
  \node[draw, circle, inner sep=1.2pt, fill=black] (a) at (0,0) {};
  \node[draw, circle, inner sep=1.2pt, fill=black] (b) at (1,0) {};
  \node[draw, circle, inner sep=1.2pt, fill=black] (c) at (0.5,0.866) {};
  \draw (a) -- (b);
  \draw (a) -- (c);
  \draw (b) -- (c);
  \node[below=2pt of a] {2};
  \node[below=2pt of b] {3};
  \node[above=2pt of c] {1};
\end{tikzpicture}
& \raisebox{0.3cm}[0.8cm][0.8cm]{$\{1,2,3\}$} 
& \raisebox{0.3cm}[0.8cm][0.8cm]{$\{1,2,3\}$}  \\ 
\hline

\raisebox{-0.2cm}[0.8cm][0.8cm]{4} 
& \begin{tikzpicture}[baseline=-0.5ex, scale=1.0]
  \node[draw, circle, inner sep=1.2pt, fill=black] (a) at (0,0) {};
  \node[draw, circle, inner sep=1.2pt, fill=black] (b) at (1.2,0) {};
  \node[draw, circle, inner sep=1.2pt, fill=black] (c) at (2.4,0) {};
  \node[draw, circle, inner sep=1.2pt, fill=black] (d) at (3.6,0) {};
  \draw (a) -- (b);
  \draw (b) -- (c);
  \draw (c) -- (d);
  \node[below=2pt of a] {1};
  \node[below=2pt of b] {2};
  \node[below=2pt of c] {3};
  \node[below=2pt of d] {4};
\end{tikzpicture}
& \raisebox{-0.2cm}[0.8cm][0.8cm]{$\{ 2,3 \}$} 
& \raisebox{-0.2cm}[0.8cm][0.8cm]{$\{ 2,3 \}$} \\
\hline

\raisebox{0.5cm}[0.8cm][0.8cm]{4}
& \begin{tikzpicture}[baseline=-0.5ex, scale=1.0]
  \node[draw, circle, inner sep=1.2pt, fill=black] (a) at (0,0) {};
  \node[draw, circle, inner sep=1.2pt, fill=black] (b) at (1.5,0) {};
  \node[draw, circle, inner sep=1.2pt, fill=black] (c) at (1.5,1.2) {};
  \node[draw, circle, inner sep=1.2pt, fill=black] (d) at (0,1.2) {};
  \draw (a) -- (b);
  \draw (a) -- (c);
  \draw (a) -- (d);
  \draw (b) -- (c);
  \draw (b) -- (d);
  \draw (c) -- (d);
  \node[below=2pt of a] {1};
  \node[below=2pt of b] {2};
  \node[above=2pt of c] {3};
  \node[above=2pt of d] {4};
\end{tikzpicture}
& \raisebox{0.5cm}[0.8cm][0.8cm]{$\{1,2,3,4\}$} 
& \raisebox{0.5cm}[0.8cm][0.8cm]{$\{1,2,3,4\}$} \\
\hline

\raisebox{0.1cm}[0.8cm][0.8cm]{5}
& \begin{tikzpicture}[baseline=-0.5ex, scale=0.9]
  \def\r{1.0}
  \foreach \i/\name in {1/a,2/b,3/c,4/d,5/e} {
    \node[draw, circle, inner sep=1.2pt, fill=black] (\name) at ({90 + (\i - 1)*72}:\r) {};
  }
  \foreach \i/\j in {a/b, b/c, c/d, d/e, e/a} {
    \draw (\i) -- (\j);
  }
  \foreach \i/\name in {1/a,2/b,3/c,4/d,5/e} {
    \node at ($(0,0)!1.25!(\name)$) {\i};
  }
\end{tikzpicture}

& \raisebox{0.1cm}[0.8cm][0.8cm]{$\{1,2,3,4,5\}$} 
& \raisebox{0.1cm}[0.8cm][0.8cm]{$\{1,2,3,4,5\}$} \\
\hline

\raisebox{0pt}[0.8cm][0.8cm]{5}
& \begin{tikzpicture}[baseline=-0.5ex, scale=0.9]
  \def\r{1.0}
  \foreach \i/\name in {1/a,2/b,3/c,4/d,5/e} {
    \node[draw, circle, inner sep=1.2pt, fill=black] (\name) at ({90 + (\i - 1)*72}:\r) {};
  }
  \foreach \i/\j in {a/b, b/c, c/d, d/e, e/a, b/e} {
    \draw (\i) -- (\j);
  }
  \foreach \i/\name in {1/a,2/b,3/c,4/d,5/e} {
    \node at ($(0,0)!1.25!(\name)$) {\i};
  }
\end{tikzpicture}
& \raisebox{0pt}[0.8cm][0.8cm]{$\{1,2,5\}$} 
& \raisebox{0pt}[0.8cm][0.8cm]{$\{1,2,5\}$} \\
\hline

\raisebox{0pt}[0.8cm][0.8cm]{5}
& \begin{tikzpicture}[baseline=-0.5ex, scale=1]
  \foreach \i/\x/\y in {
    1/0/0, 2/1.2/0, 3/2.4/0, 4/3.6/0.6, 5/3.6/-0.6
  } {
    \node[draw, circle, inner sep=1.2pt, fill=black] (x\i) at (\x,\y) {};
  }
  \foreach \i/\j in {1/2, 2/3, 3/4, 3/5, 4/5} {
    \draw (x\i) -- (x\j);
  }
  \foreach \i/\x/\y in {
    1/0/-0.3, 2/1.2/-0.3, 3/2.4/-0.3, 4/3.6/0.9, 5/3.6/-0.9
  } {
    \node at (\x,\y) {\i};
  }
\end{tikzpicture}
& \raisebox{0pt}[0.8cm][0.8cm]{$\{2, 3, 4, 5\}$} 
& \raisebox{0pt}[0.8cm][0.8cm]{$\{2, 3, 4, 5\}$} \\
\hline

\raisebox{0.5cm}[0.8cm][0.8cm]{5}
& \begin{tikzpicture}[baseline=-0.5ex, scale=1.2]
  \foreach \i/\x/\y in {
    1/0/0, 2/1.2/0, 3/2.4/0, 4/2.4/1.2, 5/1.2/1.2
  } {
    \node[draw, circle, inner sep=1.2pt, fill=black] (x\i) at (\x,\y) {};
  }
  \foreach \i/\j in {1/2, 2/3, 2/5, 3/5, 3/4, 4/5} {
    \draw (x\i) -- (x\j);
  }
  \foreach \i/\x/\y in {
    1/0/-0.3, 2/1.2/-0.3, 3/2.4/-0.3, 4/2.4/1.5, 5/1.2/1.5
  } {
    \node at (\x,\y) {\i};
  }
\end{tikzpicture}
& \raisebox{0.5cm}[0.8cm][0.8cm]{$\{2, 3, 5\}$} 
& \raisebox{0.5cm}[0.8cm][0.8cm]{$\{2, 3, 5\}$} \\
\hline

\raisebox{0.1cm}[0.8cm][0.8cm]{5}
& \begin{tikzpicture}[baseline=-0.5ex, scale=0.9]
  \def\r{1.0}
  \foreach \i/\name in {1/a,2/b,3/c,4/d,5/e} {
    \node[draw, circle, inner sep=1.2pt, fill=black] (\name) at ({90 + (\i - 1)*72}:\r) {};
  }
  \foreach \i/\j in {a/b, a/c, a/d, a/e, b/c, b/d, b/e, c/d, c/e, d/e} {
    \draw (\i) -- (\j);
  }
  \foreach \i/\name in {1/a,2/b,3/c,4/d,5/e} {
    \node at ($(0,0)!1.25!(\name)$) {\i};
  }
\end{tikzpicture}
& \raisebox{0.1cm}[0.8cm][0.8cm]{$\{1, 2, 3, 4, 5\}$} 
& \raisebox{0.1cm}[0.8cm][0.8cm]{$\{1, 2, 3, 4, 5\}$} \\
\hline

\raisebox{-1.2cm}[0.8cm][0.8cm]{6}
& \begin{tikzpicture}[baseline=-0.5ex, scale=0.6]
  \node[draw, circle, inner sep=1.2pt, fill=black] (x1) at (1, 0) {};
  \node[draw, circle, inner sep=1.2pt, fill=black] (x2) at (1, -1.5) {};
  \node[draw, circle, inner sep=1.2pt, fill=black] (x3) at (0, -3) {};
  \node[draw, circle, inner sep=1.2pt, fill=black] (x4) at (2, -3) {};
  \node[draw, circle, inner sep=1.2pt, fill=black] (x5) at (0, -4.5) {};
  \node[draw, circle, inner sep=1.2pt, fill=black] (x6) at (2, -4.5) {};
  
  \foreach \i/\j in {1/2, 2/3, 2/4, 3/5, 4/6} {
    \draw (x\i) -- (x\j);
  }
  \node at (0.6, 0) {1};
  \node at (0.6, -1.5) {2};
  \node at (-0.4, -3) {3};
  \node at (2.4, -3) {4};
  \node at (-0.4, -4.5) {5};
  \node at (2.4, -4.5) {6};
\end{tikzpicture}
& \raisebox{-1.2cm}[0.8cm][0.8cm]{$\{2, 3, 4\}$} 
& \raisebox{-1.2cm}[0.8cm][0.8cm]{$\{2, 3, 4\}$} \\
\hline

\raisebox{-1.2cm}[0.8cm][0.8cm]{6}
& \begin{tikzpicture}[baseline=-0.5ex, scale=0.6]
  \node[draw, circle, inner sep=1.2pt, fill=black] (x1) at (1, 0) {};
  \node[draw, circle, inner sep=1.2pt, fill=black] (x2) at (1, -1.5) {};
  \node[draw, circle, inner sep=1.2pt, fill=black] (x3) at (0, -3) {};
  \node[draw, circle, inner sep=1.2pt, fill=black] (x4) at (2, -3) {};
  \node[draw, circle, inner sep=1.2pt, fill=black] (x5) at (0, -4.5) {};
  \node[draw, circle, inner sep=1.2pt, fill=black] (x6) at (2, -4.5) {};
  
  \foreach \i/\j in {1/2, 2/3, 2/4, 3/4, 3/5, 4/6} {
    \draw (x\i) -- (x\j);
  }
  \node at (0.6, 0) {1};
  \node at (0.6, -1.5) {2};
  \node at (-0.4, -3) {3};
  \node at (2.4, -3) {4};
  \node at (-0.4, -4.5) {5};
  \node at (2.4, -4.5) {6};
\end{tikzpicture}
& \raisebox{-1.2cm}[0.8cm][0.8cm]{$\{2, 3, 4\}$} 
& \raisebox{-1.2cm}[0.8cm][0.8cm]{$\{2, 3, 4\}$} \\
\hline
\raisebox{-0.5cm}[0.8cm][0.8cm]{6}
& \begin{tikzpicture}[baseline=-0.5ex, scale=0.6]
  \node[draw, circle, inner sep=1.2pt, fill=black] (x1) at (0, 0) {};
  \node[draw, circle, inner sep=1.2pt, fill=black] (x2) at (2, 0) {};
  \node[draw, circle, inner sep=1.2pt, fill=black] (x3) at (4, 0) {};
  \node[draw, circle, inner sep=1.2pt, fill=black] (x4) at (4, -2) {};
  \node[draw, circle, inner sep=1.2pt, fill=black] (x5) at (2, -2) {};
  \node[draw, circle, inner sep=1.2pt, fill=black] (x6) at (0, -2) {};
  
  \foreach \i/\j in {1/2, 2/3, 2/5, 3/4, 4/5, 5/6} {
    \draw (x\i) -- (x\j);
  }
  \node at (0, 0.4) {1};
  \node at (2, 0.4) {2};
  \node at (4, 0.4) {3};
  \node at (4, -2.5) {4};
  \node at (2, -2.5) {5};
  \node at (0, -2.5) {6};
\end{tikzpicture}
& \raisebox{-0.5cm}[0.8cm][0.8cm]{$\{2, 5\}$} 
& \raisebox{-0.5cm}[0.8cm][0.8cm]{$\{2, 5\}$} \\
\hline

\raisebox{0pt}[0.8cm][0.8cm]{6}
&\begin{tikzpicture}[baseline=-0.5ex, scale=0.5]
  \node[draw, circle, inner sep=1.2pt, fill=black] (x1) at (0, 2) {};
  \node[draw, circle, inner sep=1.2pt, fill=black] (x2) at (-1.73, 1) {};
  \node[draw, circle, inner sep=1.2pt, fill=black] (x3) at (-1.73, -1) {};
  \node[draw, circle, inner sep=1.2pt, fill=black] (x4) at (0, -2) {};
  \node[draw, circle, inner sep=1.2pt, fill=black] (x5) at (1.73, -1) {};
  \node[draw, circle, inner sep=1.2pt, fill=black] (x6) at (1.73, 1) {};
  
  \foreach \i/\j in {1/2, 1/6, 2/3, 2/6, 3/4, 3/5, 4/5} {
    \draw (x\i) -- (x\j);
  }
  \node at (0, 2.6) {1};
  \node at (-2.2, 1) {2};
  \node at (-2.2, -1) {3};
  \node at (0, -2.6) {4};
  \node at (2.2, -1) {5};
  \node at (2.2, 1) {6};
\end{tikzpicture}
&\raisebox{0pt}[0.8cm][0.8cm]{$\{1,2, 3, 4, 5, 6\}$} 
& \raisebox{0pt}[0.8cm][0.8cm]{$\{1, 2, 3, 4, 5, 6\}$} \\
\hline

\raisebox{0pt}[0.8cm][0.8cm]{6}
&\begin{tikzpicture}[baseline=-0.5ex, scale=0.5]
  \node[draw, circle, inner sep=1.2pt, fill=black] (x1) at (0, 2) {};
  \node[draw, circle, inner sep=1.2pt, fill=black] (x2) at (-1.73, 1) {};
  \node[draw, circle, inner sep=1.2pt, fill=black] (x3) at (-1.73, -1) {};
  \node[draw, circle, inner sep=1.2pt, fill=black] (x4) at (0, -2) {};
  \node[draw, circle, inner sep=1.2pt, fill=black] (x5) at (1.73, -1) {};
  \node[draw, circle, inner sep=1.2pt, fill=black] (x6) at (1.73, 1) {};
  
  \foreach \i/\j in {1/2, 1/6, 2/3, 2/6, 3/4, 3/5, 4/5, 5/6} {
    \draw (x\i) -- (x\j);
  }
  \node at (0, 2.6) {1};
  \node at (-2.2, 1) {2};
  \node at (-2.2, -1) {3};
  \node at (0, -2.6) {4};
  \node at (2.2, -1) {5};
  \node at (2.2, 1) {6};
\end{tikzpicture}
&\raisebox{0pt}[0.8cm][0.8cm]{$\{1,2, 3, 4, 5, 6\}$} 
& \raisebox{0pt}[0.8cm][0.8cm]{$\{1, 2, 3, 4, 5, 6\}$} \\
\hline

\raisebox{0pt}[0.8cm][0.8cm]{6}
&\begin{tikzpicture}[baseline=-0.5ex, scale=0.5]
  \node[draw, circle, inner sep=1.2pt, fill=black] (x1) at (0, 2) {};
  \node[draw, circle, inner sep=1.2pt, fill=black] (x2) at (-1.73, 1) {};
  \node[draw, circle, inner sep=1.2pt, fill=black] (x3) at (-1.73, -1) {};
  \node[draw, circle, inner sep=1.2pt, fill=black] (x4) at (0, -2) {};
  \node[draw, circle, inner sep=1.2pt, fill=black] (x5) at (1.73, -1) {};
  \node[draw, circle, inner sep=1.2pt, fill=black] (x6) at (1.73, 1) {};
  
  \foreach \i/\j in {1/2, 1/3, 1/6, 2/3, 2/4, 3/4, 4/5, 5/6} {
    \draw (x\i) -- (x\j);
  }
  \node at (0, 2.6) {1};
  \node at (-2.2, 1) {2};
  \node at (-2.2, -1) {3};
  \node at (0, -2.6) {4};
  \node at (2.2, -1) {5};
  \node at (2.2, 1) {6};
\end{tikzpicture}
&\raisebox{0pt}[0.8cm][0.8cm]{$\{1,2, 3, 4, 5, 6\}$} 
& \raisebox{0pt}[0.8cm][0.8cm]{$\{1, 2, 3 ,4, 5, 6\}$} \\
\hline

\raisebox{0pt}[0.8cm][0.8cm]{6}
&\begin{tikzpicture}[baseline=-0.5ex, scale=0.5]
  \node[draw, circle, inner sep=1.2pt, fill=black] (x1) at (0, 2) {};
  \node[draw, circle, inner sep=1.2pt, fill=black] (x2) at (-1.73, 1) {};
  \node[draw, circle, inner sep=1.2pt, fill=black] (x3) at (-1.73, -1) {};
  \node[draw, circle, inner sep=1.2pt, fill=black] (x4) at (0, -2) {};
  \node[draw, circle, inner sep=1.2pt, fill=black] (x5) at (1.73, -1) {};
  \node[draw, circle, inner sep=1.2pt, fill=black] (x6) at (1.73, 1) {};
  
  \foreach \i/\j in {1/2, 1/6, 2/6, 3/4, 3/5, 3/6, 4/5, 5/6} {
    \draw (x\i) -- (x\j);
  }
  \node at (0, 2.6) {1};
  \node at (-2.2, 1) {2};
  \node at (-2.2, -1) {3};
  \node at (0, -2.6) {4};
  \node at (2.2, -1) {5};
  \node at (2.2, 1) {6};
\end{tikzpicture}
&\raisebox{0pt}[0.8cm][0.8cm]{$\{1,2, 3, 5, 6\}$} 
& \raisebox{0pt}[0.8cm][0.8cm]{$\{1, 2, 3, 5, 6\}$} \\
\hline

\raisebox{0.5cm}[0.8cm][0.8cm]{6}
& \begin{tikzpicture}[baseline=-0.5ex, scale=1]
  \foreach \i/\x/\y in {
    1/-2/0, 2/-1/0, 3/0/0, 4/1.2/0, 5/1.2/1.2, 6/0/1.2
  } {
    \node[draw, circle, inner sep=1.2pt, fill=black] (x\i) at (\x,\y) {};
  }
  \foreach \i/\j in {1/2, 2/3, 3/4, 3/5, 3/6, 4/5, 4/6, 5/6} {
    \draw (x\i) -- (x\j);
  }
  \foreach \i/\x/\y in {
    1/-2/-0.3, 2/-1/-0.3, 3/0/-0.3, 4/1.2/-0.3, 5/1.2/1.5, 6/0/1.5
  } {
    \node at (\x,\y) {\i};
  }
\end{tikzpicture}
& \raisebox{0.5cm}[0.8cm][0.8cm]{$\{2, 3, 4, 5, 6\}$} 
& \raisebox{0.5cm}[0.8cm][0.8cm]{$\{2, 3, 4, 5, 6\}$} \\
\hline

\raisebox{0pt}[0.8cm][0.8cm]{6}
&\begin{tikzpicture}[baseline=-0.5ex, scale=0.5]
  \node[draw, circle, inner sep=1.2pt, fill=black] (x1) at (0, 2) {};
  \node[draw, circle, inner sep=1.2pt, fill=black] (x2) at (-1.73, 1) {};
  \node[draw, circle, inner sep=1.2pt, fill=black] (x3) at (-1.73, -1) {};
  \node[draw, circle, inner sep=1.2pt, fill=black] (x4) at (0, -2) {};
  \node[draw, circle, inner sep=1.2pt, fill=black] (x5) at (1.73, -1) {};
  \node[draw, circle, inner sep=1.2pt, fill=black] (x6) at (1.73, 1) {};
  
  \foreach \i/\j in {1/2, 1/4, 1/6, 2/3, 2/6, 3/4, 3/5, 4/5, 5/6} {
    \draw (x\i) -- (x\j);
  }
  \node at (0, 2.6) {1};
  \node at (-2.2, 1) {2};
  \node at (-2.2, -1) {3};
  \node at (0, -2.6) {4};
  \node at (2.2, -1) {5};
  \node at (2.2, 1) {6};
\end{tikzpicture}
&\raisebox{0pt}[0.8cm][0.8cm]{$\{1,2, 3, 4, 5, 6\}$} 
& \raisebox{0pt}[0.8cm][0.8cm]{$\{1, 2, 3, 4, 5, 6\}$} \\
\hline

\raisebox{0pt}[0.8cm][0.8cm]{6}
&\begin{tikzpicture}[baseline=-0.5ex, scale=0.5]
  \node[draw, circle, inner sep=1.2pt, fill=black] (x1) at (0, 2) {};
  \node[draw, circle, inner sep=1.2pt, fill=black] (x2) at (-1.73, 1) {};
  \node[draw, circle, inner sep=1.2pt, fill=black] (x3) at (-1.73, -1) {};
  \node[draw, circle, inner sep=1.2pt, fill=black] (x4) at (0, -2) {};
  \node[draw, circle, inner sep=1.2pt, fill=black] (x5) at (1.73, -1) {};
  \node[draw, circle, inner sep=1.2pt, fill=black] (x6) at (1.73, 1) {};
  
  \foreach \i/\j in {1/2, 1/5, 1/6, 2/3, 2/6, 3/4, 3/6, 4/5, 5/6} {
    \draw (x\i) -- (x\j);
  }
  \node at (0, 2.6) {1};
  \node at (-2.2, 1) {2};
  \node at (-2.2, -1) {3};
  \node at (0, -2.6) {4};
  \node at (2.2, -1) {5};
  \node at (2.2, 1) {6};
\end{tikzpicture}
&\raisebox{0pt}[0.8cm][0.8cm]{$\{1,2, 3, 5, 6\}$} 
& \raisebox{0pt}[0.8cm][0.8cm]{$\{1, 2, 3, 5, 6\}$} \\
\hline

\raisebox{0pt}[0.8cm][0.8cm]{6}
&\begin{tikzpicture}[baseline=-0.5ex, scale=0.5]
  \node[draw, circle, inner sep=1.2pt, fill=black] (x1) at (0, 2) {};
  \node[draw, circle, inner sep=1.2pt, fill=black] (x2) at (-1.73, 1) {};
  \node[draw, circle, inner sep=1.2pt, fill=black] (x3) at (-1.73, -1) {};
  \node[draw, circle, inner sep=1.2pt, fill=black] (x4) at (0, -2) {};
  \node[draw, circle, inner sep=1.2pt, fill=black] (x5) at (1.73, -1) {};
  \node[draw, circle, inner sep=1.2pt, fill=black] (x6) at (1.73, 1) {};
  
  \foreach \i/\j in {1/2, 1/5, 1/6, 2/3, 2/6, 3/4, 3/5, 4/5, 5/6} {
    \draw (x\i) -- (x\j);
  }
  \node at (0, 2.6) {1};
  \node at (-2.2, 1) {2};
  \node at (-2.2, -1) {3};
  \node at (0, -2.6) {4};
  \node at (2.2, -1) {5};
  \node at (2.2, 1) {6};
\end{tikzpicture}
&\raisebox{0pt}[0.8cm][0.8cm]{$\{1, 3, 4, 5, 6\}$} 
& \raisebox{0pt}[0.8cm][0.8cm]{$\{1, 3, 4, 5, 6\}$} \\
\hline

\raisebox{0pt}[0.8cm][0.8cm]{6}
&\begin{tikzpicture}[baseline=-0.5ex, scale=0.5]
  \node[draw, circle, inner sep=1.2pt, fill=black] (x1) at (0, 2) {};
  \node[draw, circle, inner sep=1.2pt, fill=black] (x2) at (-1.73, 1) {};
  \node[draw, circle, inner sep=1.2pt, fill=black] (x3) at (-1.73, -1) {};
  \node[draw, circle, inner sep=1.2pt, fill=black] (x4) at (0, -2) {};
  \node[draw, circle, inner sep=1.2pt, fill=black] (x5) at (1.73, -1) {};
  \node[draw, circle, inner sep=1.2pt, fill=black] (x6) at (1.73, 1) {};
  
  \foreach \i/\j in {1/2, 1/6, 2/3, 2/5, 2/6, 3/4, 3/5, 4/5, 5/6} {
    \draw (x\i) -- (x\j);
  }
  \node at (0, 2.6) {1};
  \node at (-2.2, 1) {2};
  \node at (-2.2, -1) {3};
  \node at (0, -2.6) {4};
  \node at (2.2, -1) {5};
  \node at (2.2, 1) {6};
\end{tikzpicture}
&\raisebox{0pt}[0.8cm][0.8cm]{$\{2, 3, 5, 6\}$} 
& \raisebox{0pt}[0.8cm][0.8cm]{$\{2, 3, 5, 6\}$} \\
\hline

\raisebox{0pt}[0.8cm][0.8cm]{6}
&\begin{tikzpicture}[baseline=-0.5ex, scale=0.5]
  \node[draw, circle, inner sep=1.2pt, fill=black] (x1) at (0, 2) {};
  \node[draw, circle, inner sep=1.2pt, fill=black] (x2) at (-1.73, 1) {};
  \node[draw, circle, inner sep=1.2pt, fill=black] (x3) at (-1.73, -1) {};
  \node[draw, circle, inner sep=1.2pt, fill=black] (x4) at (0, -2) {};
  \node[draw, circle, inner sep=1.2pt, fill=black] (x5) at (1.73, -1) {};
  \node[draw, circle, inner sep=1.2pt, fill=black] (x6) at (1.73, 1) {};
  
  \foreach \i/\j in {1/2, 1/5, 1/6, 2/3, 2/5, 2/6, 3/4, 3/5, 5/6} {
    \draw (x\i) -- (x\j);
  }
  \node at (0, 2.6) {1};
  \node at (-2.2, 1) {2};
  \node at (-2.2, -1) {3};
  \node at (0, -2.6) {4};
  \node at (2.2, -1) {5};
  \node at (2.2, 1) {6};
\end{tikzpicture}
&\raisebox{0pt}[0.8cm][0.8cm]{$\{1, 2, 3, 5, 6\}$} 
& \raisebox{0pt}[0.8cm][0.8cm]{$\{1, 2, 3, 5, 6\}$} \\
\hline

\raisebox{0pt}[0.8cm][0.8cm]{6}
&\begin{tikzpicture}[baseline=-0.5ex, scale=0.5]
  \node[draw, circle, inner sep=1.2pt, fill=black] (x1) at (0, 2) {};
  \node[draw, circle, inner sep=1.2pt, fill=black] (x2) at (-1.73, 1) {};
  \node[draw, circle, inner sep=1.2pt, fill=black] (x3) at (-1.73, -1) {};
  \node[draw, circle, inner sep=1.2pt, fill=black] (x4) at (0, -2) {};
  \node[draw, circle, inner sep=1.2pt, fill=black] (x5) at (1.73, -1) {};
  \node[draw, circle, inner sep=1.2pt, fill=black] (x6) at (1.73, 1) {};
  
  \foreach \i/\j in {1/2, 1/5, 1/6, 2/3, 2/5, 2/6, 3/4, 4/5, 5/6} {
    \draw (x\i) -- (x\j);
  }
  \node at (0, 2.6) {1};
  \node at (-2.2, 1) {2};
  \node at (-2.2, -1) {3};
  \node at (0, -2.6) {4};
  \node at (2.2, -1) {5};
  \node at (2.2, 1) {6};
\end{tikzpicture}
&\raisebox{0pt}[0.8cm][0.8cm]{$\{1, 2, 5, 6\}$} 
& \raisebox{0pt}[0.8cm][0.8cm]{$\{1, 2, 5, 6\}$} \\
\hline

\raisebox{0pt}[0.8cm][0.8cm]{6}
&\begin{tikzpicture}[baseline=-0.5ex, scale=0.5]
  \node[draw, circle, inner sep=1.2pt, fill=black] (x1) at (0, 2) {};
  \node[draw, circle, inner sep=1.2pt, fill=black] (x2) at (-1.73, 1) {};
  \node[draw, circle, inner sep=1.2pt, fill=black] (x3) at (-1.73, -1) {};
  \node[draw, circle, inner sep=1.2pt, fill=black] (x4) at (0, -2) {};
  \node[draw, circle, inner sep=1.2pt, fill=black] (x5) at (1.73, -1) {};
  \node[draw, circle, inner sep=1.2pt, fill=black] (x6) at (1.73, 1) {};
  
  \foreach \i/\j in {1/2, 1/3, 1/6, 2/3, 2/4, 3/4, 3/5, 4/5, 4/6, 5/6} {
    \draw (x\i) -- (x\j);
  }
  \node at (0, 2.6) {1};
  \node at (-2.2, 1) {2};
  \node at (-2.2, -1) {3};
  \node at (0, -2.6) {4};
  \node at (2.2, -1) {5};
  \node at (2.2, 1) {6};
\end{tikzpicture}
&\raisebox{0pt}[0.8cm][0.8cm]{$\{2, 3, 4, 5\}$} 
& \raisebox{0pt}[0.8cm][0.8cm]{$\{3, 4\}$} \\
\hline

\raisebox{0pt}[0.8cm][0.8cm]{6}
&\begin{tikzpicture}[baseline=-0.5ex, scale=0.5]
  \node[draw, circle, inner sep=1.2pt, fill=black] (x1) at (0, 2) {};
  \node[draw, circle, inner sep=1.2pt, fill=black] (x2) at (-1.73, 1) {};
  \node[draw, circle, inner sep=1.2pt, fill=black] (x3) at (-1.73, -1) {};
  \node[draw, circle, inner sep=1.2pt, fill=black] (x4) at (0, -2) {};
  \node[draw, circle, inner sep=1.2pt, fill=black] (x5) at (1.73, -1) {};
  \node[draw, circle, inner sep=1.2pt, fill=black] (x6) at (1.73, 1) {};
  
  \foreach \i/\j in {1/2, 1/3, 1/4, 1/6, 2/3, 2/4, 3/4, 3/6, 4/5, 5/6} {
    \draw (x\i) -- (x\j);
  }
  \node at (0, 2.6) {1};
  \node at (-2.2, 1) {2};
  \node at (-2.2, -1) {3};
  \node at (0, -2.6) {4};
  \node at (2.2, -1) {5};
  \node at (2.2, 1) {6};
\end{tikzpicture}
&\raisebox{0pt}[0.8cm][0.8cm]{$\{1,2, 3, 4\}$} 
& \raisebox{0pt}[0.8cm][0.8cm]{$\{1, 3, 4\}$} \\
\hline

\raisebox{0pt}[0.8cm][0.8cm]{6}
&\begin{tikzpicture}[baseline=-0.5ex, scale=0.5]
  \node[draw, circle, inner sep=1.2pt, fill=black] (x1) at (0, 2) {};
  \node[draw, circle, inner sep=1.2pt, fill=black] (x2) at (-1.73, 1) {};
  \node[draw, circle, inner sep=1.2pt, fill=black] (x3) at (-1.73, -1) {};
  \node[draw, circle, inner sep=1.2pt, fill=black] (x4) at (0, -2) {};
  \node[draw, circle, inner sep=1.2pt, fill=black] (x5) at (1.73, -1) {};
  \node[draw, circle, inner sep=1.2pt, fill=black] (x6) at (1.73, 1) {};
  
  \foreach \i/\j in {1/2, 1/3, 1/5, 1/6, 2/3, 2/5, 2/6, 3/4, 3/5, 5/6} {
    \draw (x\i) -- (x\j);
  }
  \node at (0, 2.6) {1};
  \node at (-2.2, 1) {2};
  \node at (-2.2, -1) {3};
  \node at (0, -2.6) {4};
  \node at (2.2, -1) {5};
  \node at (2.2, 1) {6};
\end{tikzpicture}
&\raisebox{0pt}[0.8cm][0.8cm]{$\{1,2, 3, 5\}$} 
& \raisebox{0pt}[0.8cm][0.8cm]{$\{1, 2, 3, 5\}$} \\
\hline

\raisebox{0pt}[0.8cm][0.8cm]{6}
&\begin{tikzpicture}[baseline=-0.5ex, scale=0.5]
  \node[draw, circle, inner sep=1.2pt, fill=black] (x1) at (0, 2) {};
  \node[draw, circle, inner sep=1.2pt, fill=black] (x2) at (-1.73, 1) {};
  \node[draw, circle, inner sep=1.2pt, fill=black] (x3) at (-1.73, -1) {};
  \node[draw, circle, inner sep=1.2pt, fill=black] (x4) at (0, -2) {};
  \node[draw, circle, inner sep=1.2pt, fill=black] (x5) at (1.73, -1) {};
  \node[draw, circle, inner sep=1.2pt, fill=black] (x6) at (1.73, 1) {};
  
  \foreach \i/\j in {1/2, 1/3, 1/5, 1/6, 2/3, 3/4, 3/5, 4/5, 4/6, 5/6} {
    \draw (x\i) -- (x\j);
  }
  \node at (0, 2.6) {1};
  \node at (-2.2, 1) {2};
  \node at (-2.2, -1) {3};
  \node at (0, -2.6) {4};
  \node at (2.2, -1) {5};
  \node at (2.2, 1) {6};
\end{tikzpicture}
&\raisebox{0pt}[0.8cm][0.8cm]{$\{1, 3, 5\}$} 
& \raisebox{0pt}[0.8cm][0.8cm]{$\{1, 3, 5\}$} \\
\hline

\raisebox{0pt}[0.8cm][0.8cm]{6}
&\begin{tikzpicture}[baseline=-0.5ex, scale=0.5]
  \node[draw, circle, inner sep=1.2pt, fill=black] (x1) at (0, 2) {};
  \node[draw, circle, inner sep=1.2pt, fill=black] (x2) at (-1.73, 1) {};
  \node[draw, circle, inner sep=1.2pt, fill=black] (x3) at (-1.73, -1) {};
  \node[draw, circle, inner sep=1.2pt, fill=black] (x4) at (0, -2) {};
  \node[draw, circle, inner sep=1.2pt, fill=black] (x5) at (1.73, -1) {};
  \node[draw, circle, inner sep=1.2pt, fill=black] (x6) at (1.73, 1) {};
  
  \foreach \i/\j in {1/2, 1/6, 2/3, 2/5, 2/6,  3/4, 3/5, 3/6, 4/5, 5/6} {
    \draw (x\i) -- (x\j);
  }
  \node at (0, 2.6) {1};
  \node at (-2.2, 1) {2};
  \node at (-2.2, -1) {3};
  \node at (0, -2.6) {4};
  \node at (2.2, -1) {5};
  \node at (2.2, 1) {6};
\end{tikzpicture}
&\raisebox{0pt}[0.8cm][0.8cm]{$\{2, 3, 5, 6\}$} 
& \raisebox{0pt}[0.8cm][0.8cm]{$\{2, 3, 5, 6\}$} \\
\hline

\raisebox{0pt}[0.8cm][0.8cm]{6}
&\begin{tikzpicture}[baseline=-0.5ex, scale=0.5]
  \node[draw, circle, inner sep=1.2pt, fill=black] (x1) at (0, 2) {};
  \node[draw, circle, inner sep=1.2pt, fill=black] (x2) at (-1.73, 1) {};
  \node[draw, circle, inner sep=1.2pt, fill=black] (x3) at (-1.73, -1) {};
  \node[draw, circle, inner sep=1.2pt, fill=black] (x4) at (0, -2) {};
  \node[draw, circle, inner sep=1.2pt, fill=black] (x5) at (1.73, -1) {};
  \node[draw, circle, inner sep=1.2pt, fill=black] (x6) at (1.73, 1) {};
  
  \foreach \i/\j in {1/2, 1/3, 1/4, 1/5, 1/6, 2/3, 2/4, 2/5, 2/6, 3/4, 3/5, 3/6, 4/5, 4/6, 5/6} {
    \draw (x\i) -- (x\j);
  }
  \node at (0, 2.6) {1};
  \node at (-2.2, 1) {2};
  \node at (-2.2, -1) {3};
  \node at (0, -2.6) {4};
  \node at (2.2, -1) {5};
  \node at (2.2, 1) {6};
\end{tikzpicture}
&\raisebox{0pt}[0.8cm][0.8cm]{$\{1,2, 3, 4, 5, 6\}$} 
& \raisebox{0pt}[0.8cm][0.8cm]{$\{1, 2, 3, 4, 5, 6\}$} \\
\hline

\raisebox{0.5cm}[0.8cm][0.8cm]{4}
& \begin{tikzpicture}[baseline=-0.5ex, scale=1.0]
  \node[draw, circle, inner sep=1.2pt, fill=black] (a) at (0,0) {};
  \node[draw, circle, inner sep=1.2pt, fill=black] (b) at (1.5,0) {};
  \node[draw, circle, inner sep=1.2pt, fill=black] (c) at (1.5,1.2) {};
  \node[draw, circle, inner sep=1.2pt, fill=black] (d) at (0,1.2) {};
  \draw (a) -- (b);
  \draw (a) -- (d);
  \draw (b) -- (c);
  \draw (c) -- (d);
  \node[below=2pt of a] {1};
  \node[below=2pt of b] {2};
  \node[above=2pt of c] {3};
  \node[above=2pt of d] {4};
\end{tikzpicture}
& \raisebox{0.5cm}[0.8cm][0.8cm]{$\emptyset$} 
& \raisebox{0.5cm}[0.8cm][0.8cm]{$\emptyset$} \\
\hline

\raisebox{0pt}[0.8cm][0.8cm]{5}
& \begin{tikzpicture}[baseline=-0.5ex, scale=0.9]
  \def\r{1.0}
  \foreach \i/\name in {1/a,2/b,3/c,4/d,5/e} {
    \node[draw, circle, inner sep=1.2pt, fill=black] (\name) at ({90 + (\i - 1)*72}:\r) {};
  }
  \foreach \i/\j in {a/b, a/d, b/c, c/d, d/e, e/a, b/e} {
    \draw (\i) -- (\j);
  }
  \foreach \i/\name in {1/a,2/b,3/c,4/d,5/e} {
    \node at ($(0,0)!1.25!(\name)$) {\i};
  }
\end{tikzpicture}
& \raisebox{0pt}[0.8cm][0.8cm]{$\{1, 5\}$} 
& \raisebox{0pt}[0.8cm][0.8cm]{$\emptyset$} \\
\hline

\raisebox{0pt}[0.8cm][0.8cm]{6}
&\begin{tikzpicture}[baseline=-0.5ex, scale=0.5]
  \node[draw, circle, inner sep=1.2pt, fill=black] (x1) at (0, 2) {};
  \node[draw, circle, inner sep=1.2pt, fill=black] (x2) at (-1.73, 1) {};
  \node[draw, circle, inner sep=1.2pt, fill=black] (x3) at (-1.73, -1) {};
  \node[draw, circle, inner sep=1.2pt, fill=black] (x4) at (0, -2) {};
  \node[draw, circle, inner sep=1.2pt, fill=black] (x5) at (1.73, -1) {};
  \node[draw, circle, inner sep=1.2pt, fill=black] (x6) at (1.73, 1) {};
  
  \foreach \i/\j in {1/3, 1/6, 2/3, 2/6, 3/5, 4/5, 5/6} {
    \draw (x\i) -- (x\j);
  }
  \node at (0, 2.6) {1};
  \node at (-2.2, 1) {2};
  \node at (-2.2, -1) {3};
  \node at (0, -2.6) {4};
  \node at (2.2, -1) {5};
  \node at (2.2, 1) {6};
\end{tikzpicture}
&\raisebox{0pt}[0.8cm][0.8cm]{$\{5\}$} 
& \raisebox{0pt}[0.8cm][0.8cm]{$\emptyset$} \\
\hline

\raisebox{0pt}[0.8cm][0.8cm]{6}
&\begin{tikzpicture}[baseline=-0.5ex, scale=0.5]
  \node[draw, circle, inner sep=1.2pt, fill=black] (x1) at (0, 2) {};
  \node[draw, circle, inner sep=1.2pt, fill=black] (x2) at (-1.73, 1) {};
  \node[draw, circle, inner sep=1.2pt, fill=black] (x3) at (-1.73, -1) {};
  \node[draw, circle, inner sep=1.2pt, fill=black] (x4) at (0, -2) {};
  \node[draw, circle, inner sep=1.2pt, fill=black] (x5) at (1.73, -1) {};
  \node[draw, circle, inner sep=1.2pt, fill=black] (x6) at (1.73, 1) {};
  
  \foreach \i/\j in {1/2, 1/3, 1/5, 1/6, 2/3, 2/4, 3/4, 4/5, 4/6, 5/6} {
    \draw (x\i) -- (x\j);
  }
  \node at (0, 2.6) {1};
  \node at (-2.2, 1) {2};
  \node at (-2.2, -1) {3};
  \node at (0, -2.6) {4};
  \node at (2.2, -1) {5};
  \node at (2.2, 1) {6};
\end{tikzpicture}
&\raisebox{0pt}[0.8cm][0.8cm]{$\{2, 3, 5, 6\}$} 
& \raisebox{0pt}[0.8cm][0.8cm]{$\emptyset$} \\
\hline

\raisebox{0pt}[0.8cm][0.8cm]{6}
&\begin{tikzpicture}[baseline=-0.5ex, scale=0.5]
  \node[draw, circle, inner sep=1.2pt, fill=black] (x1) at (0, 2) {};
  \node[draw, circle, inner sep=1.2pt, fill=black] (x2) at (-1.73, 1) {};
  \node[draw, circle, inner sep=1.2pt, fill=black] (x3) at (-1.73, -1) {};
  \node[draw, circle, inner sep=1.2pt, fill=black] (x4) at (0, -2) {};
  \node[draw, circle, inner sep=1.2pt, fill=black] (x5) at (1.73, -1) {};
  \node[draw, circle, inner sep=1.2pt, fill=black] (x6) at (1.73, 1) {};
  
  \foreach \i/\j in {1/2, 1/4, 1/6, 2/3, 2/5, 2/6, 3/4, 3/5, 3/6, 4/5, 5/6} {
    \draw (x\i) -- (x\j);
  }
  \node at (0, 2.6) {1};
  \node at (-2.2, 1) {2};
  \node at (-2.2, -1) {3};
  \node at (0, -2.6) {4};
  \node at (2.2, -1) {5};
  \node at (2.2, 1) {6};
\end{tikzpicture}
&\raisebox{0pt}[0.8cm][0.8cm]{$\{2, 3, 5, 6\}$} 
& \raisebox{0pt}[0.8cm][0.8cm]{$\emptyset$} \\
\hline

\raisebox{0pt}[0.8cm][0.8cm]{6}
&\begin{tikzpicture}[baseline=-0.5ex, scale=0.5]
  \node[draw, circle, inner sep=1.2pt, fill=black] (x1) at (0, 2) {};
  \node[draw, circle, inner sep=1.2pt, fill=black] (x2) at (-1.73, 1) {};
  \node[draw, circle, inner sep=1.2pt, fill=black] (x3) at (-1.73, -1) {};
  \node[draw, circle, inner sep=1.2pt, fill=black] (x4) at (0, -2) {};
  \node[draw, circle, inner sep=1.2pt, fill=black] (x5) at (1.73, -1) {};
  \node[draw, circle, inner sep=1.2pt, fill=black] (x6) at (1.73, 1) {};
  
  \foreach \i/\j in {1/2, 1/3, 1/5, 1/6, 2/3, 2/4, 2/5, 3/4, 4/5, 4/6, 5/6} {
    \draw (x\i) -- (x\j);
  }
  \node at (0, 2.6) {1};
  \node at (-2.2, 1) {2};
  \node at (-2.2, -1) {3};
  \node at (0, -2.6) {4};
  \node at (2.2, -1) {5};
  \node at (2.2, 1) {6};
\end{tikzpicture}
&\raisebox{0pt}[0.8cm][0.8cm]{$\{1, 2, 5\}$} 
& \raisebox{0pt}[0.8cm][0.8cm]{$\emptyset$} \\
\hline

\raisebox{0pt}[0.8cm][0.8cm]{6}
&\begin{tikzpicture}[baseline=-0.5ex, scale=0.5]
  \node[draw, circle, inner sep=1.2pt, fill=black] (x1) at (0, 2) {};
  \node[draw, circle, inner sep=1.2pt, fill=black] (x2) at (-1.73, 1) {};
  \node[draw, circle, inner sep=1.2pt, fill=black] (x3) at (-1.73, -1) {};
  \node[draw, circle, inner sep=1.2pt, fill=black] (x4) at (0, -2) {};
  \node[draw, circle, inner sep=1.2pt, fill=black] (x5) at (1.73, -1) {};
  \node[draw, circle, inner sep=1.2pt, fill=black] (x6) at (1.73, 1) {};
  
  \foreach \i/\j in {1/2, 1/3, 1/5, 1/6, 2/4, 3/4, 3/5, 3/6, 4/5, 4/6, 5/6} {
    \draw (x\i) -- (x\j);
  }
  \node at (0, 2.6) {1};
  \node at (-2.2, 1) {2};
  \node at (-2.2, -1) {3};
  \node at (0, -2.6) {4};
  \node at (2.2, -1) {5};
  \node at (2.2, 1) {6};
\end{tikzpicture}
&\raisebox{0pt}[0.8cm][0.8cm]{$\{3, 5, 6\}$} 
& \raisebox{0pt}[0.8cm][0.8cm]{$\emptyset$} \\
\hline

\raisebox{0pt}[0.8cm][0.8cm]{6}
&\begin{tikzpicture}[baseline=-0.5ex, scale=0.5]
  \node[draw, circle, inner sep=1.2pt, fill=black] (x1) at (0, 2) {};
  \node[draw, circle, inner sep=1.2pt, fill=black] (x2) at (-1.73, 1) {};
  \node[draw, circle, inner sep=1.2pt, fill=black] (x3) at (-1.73, -1) {};
  \node[draw, circle, inner sep=1.2pt, fill=black] (x4) at (0, -2) {};
  \node[draw, circle, inner sep=1.2pt, fill=black] (x5) at (1.73, -1) {};
  \node[draw, circle, inner sep=1.2pt, fill=black] (x6) at (1.73, 1) {};
  
  \foreach \i/\j in {1/2, 1/3, 1/5, 1/6, 2/3, 2/4, 2/6, 3/4, 3/5, 4/5, 4/6, 5/6} {
    \draw (x\i) -- (x\j);
  }
  \node at (0, 2.6) {1};
  \node at (-2.2, 1) {2};
  \node at (-2.2, -1) {3};
  \node at (0, -2.6) {4};
  \node at (2.2, -1) {5};
  \node at (2.2, 1) {6};
\end{tikzpicture}
&\raisebox{0pt}[0.8cm][0.8cm]{$\emptyset$} 
&\raisebox{0pt}[0.8cm][0.8cm]{$\emptyset$} \\
\hline

\end{longtable}
\end{center}

By the lists of $U(G)$ and ${\rm Shed}(G)$, the following statement is easily verified:

\begin{rem}\label{ver-CM-S_{2}}
For a  well-covered graph $G$ with $|V(G)|\leq 6$, the following conditions are equivalent: 
\begin{enumerate}
\item $G$ is vertex decomposable. 
\item $G$ is Cohen--Macaulay.  
\item $G$ satisfies Serre's condition $(S_{2})$. 
\end{enumerate}
\end{rem}

\begin{rem}\label{U neq Shed}
We consider the following two graphs:
\begin{center}
\begin{minipage}[t]{0.3\linewidth}
  \centering
  \begin{tikzpicture}[scale=0.5]
    \node[draw, circle, inner sep=1.2pt, fill=black] (x1) at (0, 2) {};
    \node[draw, circle, inner sep=1.2pt, fill=black] (x2) at (-1.73, 1) {};
    \node[draw, circle, inner sep=1.2pt, fill=black] (x3) at (-1.73, -1) {};
    \node[draw, circle, inner sep=1.2pt, fill=black] (x4) at (0, -2) {};
    \node[draw, circle, inner sep=1.2pt, fill=black] (x5) at (1.73, -1) {};
    \node[draw, circle, inner sep=1.2pt, fill=black] (x6) at (1.73, 1) {};

    \foreach \i/\j in {1/2, 1/3, 1/6, 2/3, 2/4, 3/4, 3/5, 4/5, 4/6, 5/6} {
      \draw (x\i) -- (x\j);
    }

    \node at (0, 2.6) {1};
    \node at (-2.2, 1) {2};
    \node at (-2.2, -1) {3};
    \node at (0, -2.6) {4};
    \node at (2.2, -1) {5};
    \node at (2.2, 1) {6};
  \end{tikzpicture}
  \captionof{figure}{}
\end{minipage}
\hspace{0.1cm}
\begin{minipage}[t]{0.3\linewidth}
  \centering
  \begin{tikzpicture}[scale=0.5]
    \node[draw, circle, inner sep=1.2pt, fill=black] (x1) at (0, 2) {};
    \node[draw, circle, inner sep=1.2pt, fill=black] (x2) at (-1.73, 1) {};
    \node[draw, circle, inner sep=1.2pt, fill=black] (x3) at (-1.73, -1) {};
    \node[draw, circle, inner sep=1.2pt, fill=black] (x4) at (0, -2) {};
    \node[draw, circle, inner sep=1.2pt, fill=black] (x5) at (1.73, -1) {};
    \node[draw, circle, inner sep=1.2pt, fill=black] (x6) at (1.73, 1) {};

    \foreach \i/\j in {1/2, 1/3, 1/4, 1/6, 2/3, 2/4, 3/4, 3/6, 4/5, 5/6} {
      \draw (x\i) -- (x\j);
    }

    \node at (0, 2.6) {1};
    \node at (-2.2, 1) {2};
    \node at (-2.2, -1) {3};
    \node at (0, -2.6) {4};
    \node at (2.2, -1) {5};
    \node at (2.2, 1) {6};
  \end{tikzpicture}
  \captionof{figure}{}
\end{minipage}
\end{center}
Let $G_{1}$ denote the graph shown in Figure 1, and $G_{2}$ denote the graph shown in Figure 2. $G_{1}$ and $G_{2}$ are only Cohen--Macaulay graphs with at most 6 vertices such that $U(G)\neq {\rm Shed}(G)$. Now, for each $i=1,2$, the graph $G_{i} - 2$ is not vertex decomposable, whereas $G_{i} - N[2]$ is vertex decomposable. By symmetry, the same holds when considering vertex 5 instead of 2 in $G_{1}$.
\end{rem}

Under certain conditions, we provide a necessary and sufficient condition for $G_{0}(\mathcal{H})$ to be Cohen–Macaulay for every graph $G_{0}$ on the vertex set $X_{[n]}$. 

\begin{thm}\label{eqiv between CM and vd}
Suppose that $|V(H_{i})|\leq 6$ and assume that the condition $(\ast)$, then the following conditions are equivalent: 
\begin{enumerate}
\item $H_{i}$ is 2-pure with respect to $x_{i}$ and $H_{i}$ is vertex decomposable such that $x_{i}$ is a shedding vertex for all $i$. 
\item $G_{0}(\mathcal{H})$ is pure vertex decomposable for every graph $G_{0}$ on  $X_{[m]}$. 
\item $G_{0}(\mathcal{H})$ is Cohen--Macaulay for every graph $G_{0}$ on $X_{[m]}$. 
\item $G_{0}(\mathcal{H})$ satisfies Serre's condition $(S_{2})$ for every graph $G_{0}$ on $X_{[m]}$.  
\end{enumerate}
\end{thm}
\begin{proof}
$(1)\Rightarrow(2)$ follows from Proposition \ref{vertex decomp}. Also $(2)\Rightarrow (3)$ follows from the fact that pure vertex decomposable implies Cohen--Macaulay. $(3)\Rightarrow(4)$ always follows. 
So we must prove that the condition $(4)$ implies the condition $(1)$. We consider $G_{0}$ as the empty graph, that is, $G_{0}$ has no edges. Then we have$$\Bbbk[V(G_{0}(\mathcal{H}))]/I(G_{0}(\mathcal{H}))\simeq\Bbbk[V(H_{1})]/I(H_{1})\otimes_{\Bbbk}\cdots\otimes_{\Bbbk}\Bbbk[V(H_{n})]/I(H_{n}).$$Therefore, $\Bbbk[V(H_{i})]/I(H_{i})$ satisfies Serre's condition $(S_{2})$ for all $i$ by Lemma \ref{lemma2}. Now, since $|V(H_i)| \leq 6$, it follows that each $H_{i}$ is vertex decomposable from Remark \ref{ver-CM-S_{2}}. Hence, it suffices to prove that $x_{i}$ is a shedding vertex of $H_{i}$ for all $i$. To this end, assume for contradiction that there exists some $i$ such that $x_{i}$ is not a shedding vertex of $H_{i}$. We distinguish the following cases:

\vspace{0.1cm}
{\it Case 1}: $x_{k}$ is not a shedding vertex of $H_{k}$ for all $k$. 

\vspace{0.1cm} In this case, we consider a graph $G_{0}$ as a complete graph. Then the graph $G_{0}(\mathcal{H}) - N[x_{i}]$ consists of connected components of the form $H_{k} - x_{k}$ for all $k\neq i$, and $H_{i} - N[x_{i}]$. Since $x_{k}$ is not a shedding vertex, none of the connected components is vertex decomposable. Given that $|V(H_{k})| \leq 6$, we conclude that $H_{k}-x_{k}$ and $H_{i}-N[x_{i}]$ do not satisfy Serre's condition $(S_{2})$. Therefore, by the above isomorphism, $G_{0}(\mathcal{H})$ does not satisfy Serre's condition $(S_{2})$, which is a contradiction. 

\vspace{0.1cm}
{\it Case 2}: There exists $j$ such that $x_{j}$ is a shedding vertex of $H_{j}$. 

\vspace{0.1cm} In this case, we consider a graph $G_{0}$ such that $\{x_{i}, x_{j}\}$ is an  only edge of $G_{0}$. Then, from the proof of Proposition \ref{vertex decomp}, we deduce that $G_{0}(\mathcal{H})- x_{j}$ is also vertex decomposable, that is, $x_{j}$ is a shedding vertex of $G_{0}(\mathcal{H})$. Therefore, by the same argument as above, we have $H_{i}-x_{i}$ is vertex decomposable. From $|V(H_{i})|\leq 6$ and Remark \ref{U neq Shed}, it is a contradiction, as required. 
\end{proof}


\section{The size of Betti tables of edge ideals of multi-clique corona graphs}
In this section, we define multi-clique corona graphs as a generalization of clique corona graphs and multi-whisker graphs. We establish a relationship between multi-clique corona graphs and multi-whisker graphs from an algebraic perspective. This gives us the characterization of the regularity of edge ideals of multi-clique corona graphs. Also, we prove that multi-clique corona graphs are sequentially Cohen--Macaulay. This result leads us to the characterization of the projective dimension of multi-clique corona graphs. Let us define the multi-clique corona graphs. 

\begin{defi}
For a graph $G$ on the vertex set $X_{[h]}$ and positive integers $n_{1},\ldots, n_{h}$, we set 
$$\mathcal{H}=\bigcup_{1\leq i\leq h}\{H_{i}^{1},\ldots, H_{i}^{n_{i}}\},$$
where $H_{i}^{j}$ is a nonempty graph indexed by the vertex $x_{i}$. 
we define the {\it multi-corona graph} $G\circ\mathcal{H}$ of $G$ and $\mathcal{H}$ is the disjoint union of $G$ and $H_{i}^{1},\ldots, H_{i}^{n_{i}}$, with additional edges joining each vertex $x_{i}$ to all vertices $H_{i}^{j}$. 
We call $G\circ\mathcal{H}$ is {\it multi-clique corona graph}, if $H_{i}^{j}$ is the complete graph for all $i$ and $j$. 
\end{defi}

We establish a relationship between multi-clique corona graphs and multi-whisker graphs. To this end, we recall the definition of multi-whisker graphs, which was first introduced by Pournaki and the authors in this paper in the context of combinatorial commutative algebra \cite[Section 1]{mpt}.  

\begin{defi}\cite[Section 1]{mpt}\label{def of multi}
Let $G_{0}$ be a graph on the vertex set $V(G_{0})=X_{[r]}=\{x_{1},\ldots, x_{r}\}$ and $n_{1},\ldots, n_{r}$ be positive integers. Then the {\it multi-whisker graph} associated with $G_{0}$ is the graph $G$ on the vertex set 
$$V(G) = X_{[r]} \cup Y, \mbox{ where } Y = \{y_ {1,1}, \ldots, y_ {1, n_r} 
\} \cup \cdots \cup \{y_ {r,1}, \ldots, y_ {r, n_r}\}$$ 
and the edge set 
$$ E(G) = E(G_{0}) \cup \{x_1 y_{1,1}, \ldots, x_1 y_ {1, n_1}, \ldots, x_r y_{r,1}, \ldots, x_r y_{r, n_r} \}.$$
\end{defi}

We prepare several notations. 
\begin{nota}\label{split2}
Let $G$ be a graph on the vertex set $X_{[h]}$ and positive integers $n_{1},\ldots, n_{h}$ and let $\mathcal{H}=\{H_{i}^{1},\ldots, H_{i}^{n_{i}}\,\,: x_{i}\in V(G)\}$. Assume that $G\circ\mathcal{H}$ is a multi-clique corona graph. Let $H_{i}^{j}=K_{m_{i,j}}$, where $K_{m_{i,j}}$ is the complete graph with $m_{i,j}$ vertices. Set $V(H_{i}^{j})=\{y_{m_{i,j}, 1},\ldots, y_{m_{i,j}, m_{i,j}}\}$. Also, we set the following new vertex set 
$$Z=\bigcup_{1\leq i\leq h}\{z_{m_{i,1},0},\ldots, z_{m_{i,1},m_{i,1}-1},\ldots, z_{m_{i,n_{i}},0},\ldots, z_{m_{i, n_{i}}, m_{i, n_{i}}-1}\}$$
We consider the graph $G^{\prime}$ such that 
$$V(G^{\prime})=(V(G)-\{y_{m_{i,1},m_{i,1}},\ldots, y_{m_{h,n_{h}},m_{h,n_{h}}}\})\cup Z,$$ and 
$$E(G^{\prime})=E(G-\{y_{m_{i,1},m_{i,1}},\ldots, y_{m_{h,n_{h}},m_{h,n_{h}}}\})\cup\left(\bigcup_{1\leq i\leq h}\{\{y_{m_{i,j},k}, z_{m_{i,j},k}\}, \{x_{i}, z_{m_{i, j}, 0}\}\}
\right).$$
Notice that the graph $G^{\prime}$ is the multi-whisker graphs. 
Also, we let $S=\Bbbk[V(G\circ\mathcal{H})]$ and let $S^{\prime}=\Bbbk[V(G^{\prime})].$ Then the edge ideal $I(G\circ\mathcal{H})$ is an ideal of $S$ and $I(G^{\prime})$ is an ideal of $S^{\prime}$. 
\end{nota}

\begin{ex}
Let $G=P_{2}$ be the path on $\{x_{1}, x_{2}\}$. We consider $\mathcal{H}$ as $H_{1}^{1}=K_{2}, H_{1}^{2}=K_{3}, H_{2}^{1}=K_{1}$ and $H_{2}^{2}=K_{2}$. Then the graph $G\circ\mathcal{H}$ and $G^{\prime}$ provided in Construction \ref{split2} are as follows: 

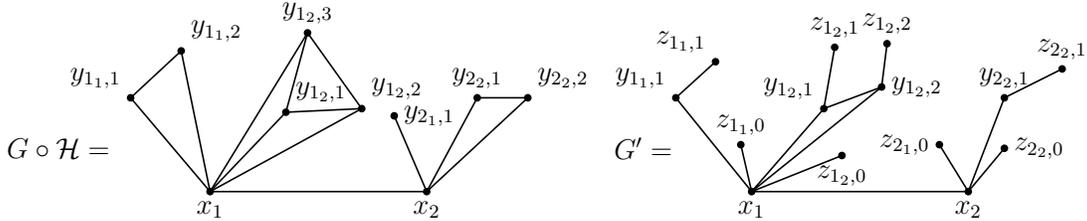
\begin{center}
\begin{tikzpicture}[scale=0.48, every node/.style={scale=0.9}]
\node at (-4.2,1.2) {$G\circ\mathcal{H}=$};

\coordinate (X1) at (0,0);
\coordinate (X2) at (6,0);
\fill (X1) circle (3pt) node[below] {$x_{1}$};
\fill (X2) circle (3pt) node[below] {$x_{2}$};
\draw[semithick] (X1)--(X2);

\coordinate (A1) at (-2.2,2.6);  
\coordinate (A2) at (-0.8,3.9);  
\fill (A1) circle (3pt) node[above left] {$y_{1_{1},1}$};
\fill (A2) circle (3pt) node[above right] {$y_{1_{1},2}$};
\draw[semithick] (A1)--(A2);
\draw[semithick] (X1)--(A1);
\draw[semithick] (X1)--(A2);

\coordinate (B1) at (2.1,2.2);   
\coordinate (B2) at (4.2,2.3);   
\coordinate (B3) at (2.7,4.4);   
\fill (B1) circle (3pt) node[above right]  {$y_{1_{2},1}$};
\fill (B2) circle (3pt) node[above right] {$y_{1_{2},2}$};
\fill (B3) circle (3pt) node[above]       {$y_{1_{2},3}$};
\draw[semithick] (B1)--(B2)--(B3)--(B1);
\draw[semithick] (X1)--(B1);
\draw[semithick] (X1)--(B2);
\draw[semithick] (X1)--(B3);

\coordinate (C1) at (5.1,2.1);   
\fill (C1) circle (3pt) node[right] {$y_{2_{1},1}$};
\draw[semithick] (X2)--(C1);

\coordinate (D1) at (7.4,2.6);   
\coordinate (D2) at (8.8,2.6);   
\fill (D1) circle (3pt) node[above] {$y_{2_{2},1}$};
\fill (D2) circle (3pt) node[above right] {$y_{2_{2},2}$};
\draw[semithick] (D1)--(D2);
\draw[semithick] (X2)--(D1);
\draw[semithick] (X2)--(D2);

\begin{scope}[shift={(-2,0)}]
\node at (14,1.2) {$G^{\prime}=$};

\coordinate (X1R) at (17,0);
\coordinate (X2R) at (23,0);
\fill (X1R) circle (3pt) node[below] {$x_{1}$};
\fill (X2R) circle (3pt) node[below] {$x_{2}$};
\draw[semithick] (X1R)--(X2R);

\coordinate (A1R)   at (14.9,2.6);   
\coordinate (AZone)  at (16,3.6);   
\coordinate (AZzero) at (16.7,1.3);   
\fill (A1R) circle (3pt) node[above left] {$y_{1_{1},1}$};
\fill (AZone) circle (3pt) node[above left] {$z_{1_{1},1}$};
\fill (AZzero) circle (3pt) node[above] {$z_{1_{1},0}$};
\draw[semithick] (X1R)--(A1R);
\draw[semithick] (A1R)--(AZone);
\draw[semithick] (X1R)--(AZzero);

\coordinate (B1R)   at (19.0,2.3);   
\coordinate (B2R)   at (20.6,2.9);   
\coordinate (BZone) at (19.3,4.0);   
\coordinate (BZtwo) at (20.75,4.1);   
\coordinate (BZzero)at (19.5,1);  
\fill (B1R) circle (3pt) node[above left]  {$y_{1_{2},1}$};
\fill (B2R) circle (3pt) node[right] {$y_{1_{2},2}$};
\fill (BZone) circle (3pt) node[above] {$z_{1_{2},1}$};
\fill (BZtwo) circle (3pt) node[above] {$z_{1_{2},2}$};
\fill (BZzero) circle (3pt) node[below] {$z_{1_{2},0}$};
\draw[semithick] (B1R)--(B2R);
\draw[semithick] (X1R)--(B1R);
\draw[semithick] (X1R)--(B2R);
\draw[semithick] (B1R)--(BZone);
\draw[semithick] (B2R)--(BZtwo);
\draw[semithick] (X1R)--(BZzero);

\coordinate (CZzero) at (22.2,1.3);  
\fill (CZzero) circle (3pt) node[left] {$z_{2_{1},0}$};
\draw[semithick] (X2R)--(CZzero);

\coordinate (D1R)   at (24,2.6);   
\coordinate (DZone)  at (25.6,3.4);  
\coordinate (DZzero) at (24.0,1.2);  
\fill (D1R) circle (3pt) node[above] {$y_{2_{2},1}$};
\fill (DZone) circle (3pt) node[above] {$z_{2_{2},1}$};
\fill (DZzero) circle (3pt) node[right] {$z_{2_{2},0}$};
\draw[semithick] (X2R)--(D1R);
\draw[semithick] (D1R)--(DZone);
\draw[semithick] (X2R)--(DZzero);
\end{scope}
\end{tikzpicture}
\end{center}
\end{ex}

The following theorem gives us the graded Betti numbers of $I(G\circ\mathcal{H})$ via $I(G^{\prime})$. 

\begin{thm}\label{betti}
Let $G$ be a graph on the vertex set $X_{[h]}$ and positive integers $n_{1},\ldots, n_{h}$ and let $\mathcal{H}=\bigcup_{1\leq i\leq h}\{H_{i}^{1},\ldots, H_{i}^{n_{i}}\}$. Assume that $G\circ\mathcal{H}$ is the multi-clique corona graph and the condition in Notation \ref{split2}. Then, we have 
$$\beta_{p,q}(S/I(G\circ\mathcal{H}))=\beta_{p,q}(S^{\prime}/I(G^{\prime}))\mbox{ for all }p, q.$$
\end{thm}
\begin{proof}

Notice that $x_{1}+y_{m_{1,1}, 1}+\cdots+y_{m_{1,1}, m_{1,1}}$ is a non-zero divisor of $S/I(G\circ\mathcal{H})$ since $N(y_{m_{1,1}, j})=\{x_{1}, y_{m_{1,1}, 1},\ldots, y_{m_{1,1}, m_{1,1}}\}$. Set $$G\circ\mathcal{H}^{[1]}=G\circ\mathcal{H}-\{\{u,y_{m_{1,1}, m_{1,1}}\}\in E(G\circ\mathcal{H})\,\,: u\in N(y_{m_{1,1}, m_{1,1}})\}.$$ 
Then, we have the following isomorphism: 
\begin{align*}
&S/(I(G\circ\mathcal{H})+(x_{1}+y_{m_{1,1}, 1}+\cdots+y_{m_{1,1}, m_{1,1}})) \\
&\simeq\Bbbk[V(G\circ\mathcal{H})\setminus\{y_{m_{1,1}, m_{1,1}}\}]/(I(G\circ\mathcal{H}^{[1]})+(x_{1}^2,y_{m_{1,1}, 1}^2,\ldots,  y_{m_{1,1}, m_{1,1}-1}^2)) 
\end{align*}
Let $S^{[1]}=\Bbbk[V(G\circ\mathcal{H})\setminus\{y_{m_{1,1}, m_{1,1}}\}]$ and $I^{[1]}=I(G\circ\mathcal{H}^{[1]})+(x_{1}^2,y_{m_{1,1}, 1}^2,\ldots,  y_{m_{1,1}, m_{1,1}-1}^2)$. 
By applying polarization to the latter one, we obtain the ring $(S^{[1]})^{\rm pol}/(I^{[1]})^{\rm pol}$. Notice that 
$$(S^{[1]})^{\rm pol}=\Bbbk[(V(G\circ\mathcal{H})\setminus\{y_{m_{1,1}, m_{1,1}}\})\cup\{x_{1}^{(1)}, x_{1}^{(2)}, y_{m_{1,1}, 1}^{(1)},\cdots, y_{m_{1,1}, m_{1,1}-1}^{(2)}\}]$$ and 
$$(I^{[1]})^{\rm pol}=I(G\circ\mathcal{H})+(x_{1}^{(1)}x_{1}^{(2)}, y_{m_{1,1}, 1}^{(1)}y_{m_{1,1}, 1}^{(2)}, y_{m_{1,1}, m_{1,1}-1}^{(1)}y_{m_{1,1}, m_{1,1}-1}^{(2)}).$$ 
We identify $x_{1}$ with $x_{1}^{(1)}$ and 
$z_{m_{i,1},0}$ with $x_{1}^{(2)}$. 
Also we identify $y_{m_{1,1}, 1}^{(1)}$ with $y_{m_{1,1}, j}$ and 
$z_{m_{1,1},j}$ with $y_{m_{1,1}, j}^{(2)}$
Since polarization preserve the graded Betti number, we have $$\beta_{p,q}(S/I(G\circ\mathcal{H}))=\beta_{p,q}((S^{[1]})^{\rm pol}/(I^{[1]})^{\rm pol})\mbox{ for all }p,q.$$ 
By repeating this process and finally get $S^{\prime}/I(G^{\prime})$ and 
$$\beta_{p,q}(S/I(G\circ\mathcal{H}))=\beta_{p,q}(S^{\prime}/I(G^{\prime}))\mbox{ for all }p, q, $$
as required. 
\end{proof}

As a corollary with \cite[Corollary 4.4]{mpt}, we obtain the regularity of $I(G\circ\mathcal{H})$. 

\begin{cor}\label{reg}
Let $G$ be a graph on the vertex set $X_{[h]}$ and positive integers $n_{1},\ldots, n_{h}$ and let $\mathcal{H}=\{H_{i}^{1},\ldots, H_{i}^{n_{i}}\,\,: x_{i}\in V(G)\}$. If $G\circ\mathcal{H}$ is multi-clique corona graph, then we have
\begin{align*}
{\rm reg}(\Bbbk[G\circ\mathcal{H}]/I(G\circ\mathcal{H}))&={\rm im}(G\circ\mathcal{H}) \\
&=|\{H_{i}^{j}\,\,: H_{i}^{j}\mbox{ is not the complete graph }K_{1}\}|. 
\end{align*}
\end{cor}
\begin{proof}
By Theorem \ref{betti} and \cite[Corollary 4.4]{mpt}, we have 
$${\rm reg}(S/I(G\circ\mathcal{H}))={\rm reg}S^{\prime}/I(G^{\prime})= {\rm im}(G^{\prime})= {\rm im}(G\circ\mathcal{H}),$$which completes the proof. 
\end{proof}

The next theorem is a generalization of the result about multi-whisker graphs \cite[Theorem 6.1]{mpt}. 
\begin{thm}\label{seq CM}
Let $G$ be a graph on the vertex set $X_{[h]}$ and positive integers $n_{1},\ldots, n_{h}$ and let $\mathcal{H}=\{H_{i}^{1},\ldots, H_{i}^{n_{i}}\,\,: x_{i}\in V(G)\}$. If $G\circ\mathcal{H}$ is multi-clique corona graph, then $G\circ\mathcal{H}$ is vertex decomposable and hence, sequentially Cohen--Macaulay. 
\end{thm}
\begin{proof}
We prove by induction on $|V(G\circ\mathcal{H})|$. If $|V(G\circ\mathcal{H})|=2$, then $G\circ\mathcal{H}$ is just the complete graph $K_{2}$ and it is vertex decomposable. Assume that $|V(G\circ\mathcal{H})|>2$. If $G\circ\mathcal{H}$ is clique corona graph, then we are done by \cite[Theorem 1.1]{hhko}. Hence, we may assume that there  exists an integer ${i}$ such that $n_{i}\geq 2$. Without loss of generality, we may assume that $i=h$. Set $H_{h}^{\ell}=K_{m_{h,\ell}}$, where $K_{m_{h,\ell}}$ is the complete graph with $m_{h,\ell}$ vertices. Also, we fix a vertex $y_{m_{h,n_{h}}, m_{h,n_{h}}}\in V(K_{m_{h, n_{h}}})$. 
Then, the link ${\rm link}_{\Delta(G\circ\mathcal{H})}y_{m_{h,n_{h}}, m_{h,n_{h}}}$ is 
$$\Delta(G\circ\mathcal{H}-\{x_{h}, y_{m_{h,n_{h}},\ell}\,\,: 1\leq l\leq n_{h} \})*\Delta(K_{m_{h,1}})*\cdots*\Delta(K_{m_{h,n_{h}-1}}).$$ 
Notice that $G\circ\mathcal{H}-\{x_{h}, y_{m_{h,n_{h}},\ell}\,\,: 1\leq l\leq n_{h} \}$ is a multi-clique corona graph whose number of vertices is less than $|V(G\circ\mathcal{H})|$. Hence, by the induction hypothesis and the fact that complete graphs are vertex decomposable, ${\rm link}_{\Delta(G\circ\mathcal{H})}y_{m_{h,n_{h}}, m_{h,n_{h}}}$ is vertex decomposable. 
Also, we have ${\rm del}_{\Delta(G\circ\mathcal{H})}y_{m_{h,n_{h}}, m_{h,n_{h}}}=\Delta(G\circ\mathcal{H}-y_{j_{h},n_{h}})$ and $G\circ\mathcal{H}-y_{m_{h,n_{h}}, m_{h,n_{h}}}$ is the multi-clique corona graph with the number of vertices $<|V(G\circ\mathcal{H})|$, by the induction hypothesis, ${\rm del}_{\Delta(G\circ\mathcal{H})}y_{m_{h,n_{h}}, m_{h,n_{h}}}$ is vertex decomposable. Moreover, any independent set of $G\circ\mathcal{H}-N[y_{m_{h,n_{h}}, m_{h,n_{h}}}]$ is not a maximal independent set of $G\circ\mathcal{H}-y_{m_{h,n_{h}}, m_{h,n_{h}}}$, and hence $\mathcal{F}({\rm del}_{\Delta(G\circ\mathcal{H})}y_{m_{h,n_{h}}, m_{h,n_{h}}})\subset\mathcal{F}(\Delta(G\circ\mathcal{H}))$. Therefore, $\Delta(G\circ\mathcal{H})$ is vertex decomposable, as required. 
\end{proof}

\begin{cor}\label{pd}
Let $G$ be a graph on the vertex set $X_{[h]}$ and positive integers $n_{1},\ldots, n_{h}$ and let $\mathcal{H}=\bigcup_{1\leq i\leq h}\{H_{i}^{1},\ldots, H_{i}^{n_{i}}\}$. If $G\circ\mathcal{H}$ is multi-clique corona graph, then 
$${\rm pd}\hspace{0.05cm}\Bbbk[G\circ\mathcal{H}]/I(G\circ\mathcal{H})={\rm bight}I(G\circ\mathcal{H}).$$
\end{cor}
\begin{proof}
Notice that $F\cup N(F)=V(G\circ\mathcal{H})$ holds for any maximal independent set $F$ of $G\circ\mathcal{H}$. Therefore, By \cite[Remark 5.7]{ds} and Theorem \ref{seq CM}, the assertion follows. 
\end{proof}


\section*{Acknowledgement}
The authors would like to express their sincere gratitude to Seyed Amin Seyed Fakhari for kindly informing them of the paper \cite{mfy} on edge ideals of rooted products of graphs. They are also deeply grateful to Kyouko Kimura for her insightful suggestions, which inspired the concept of multi-clique corona graphs through discussions at several conferences. Finally, the authors would like to thank Tatsuya Kataoka for his helpful and valuable comments. The research of Yuji Muta was partially supported by ohmoto-ikueikai and JST SPRING Japan Grant Number JPMJSP2126. The research of Naoki Terai was partially supported by JSPS Grant-in Aid for Scientific Research (C) 24K06670.



\end{document}